\documentclass[11pt,reqno]{amsart}
\usepackage[a4paper]{anysize}\marginsize{2.5cm}{2.5cm}{2cm}{2cm}
\pdfpagewidth=\paperwidth \pdfpageheight=\paperheight
\allowdisplaybreaks
\usepackage[english]{babel}
\usepackage[latin1]{inputenc}
\usepackage[T1]{fontenc}
\usepackage{amssymb,amsmath,amstext,amsfonts,amsthm,mathrsfs}
\usepackage{mathtools}
\usepackage{relsize}
\usepackage{booktabs}
\usepackage[colorlinks=true, urlcolor=blue, linkcolor=blue, citecolor=blue]{hyperref}
\usepackage{caption}
\usepackage{graphicx}
\usepackage{array}
\usepackage{xcolor}
\usepackage[matrix,arrow]{xy}
\usepackage{url}
\usepackage{longtable}
\usepackage{nicematrix}
\usepackage{tikz}
\usetikzlibrary{decorations.pathreplacing,calligraphy}
\tikzset{mybrace/.style={thick, decorate, decoration={calligraphic brace, amplitude=2mm, raise=-1mm}}}
\usepackage{float}

\numberwithin{equation}{section}
\theoremstyle{plain}
\newtheorem*{theorem*}{Theorem}
\newtheorem{theorem}{Theorem}
\numberwithin{theorem}{section}
\newtheorem{proposition}[theorem]{Proposition}
\newtheorem{lemma}[theorem]{Lemma}
\newtheorem{corollary}[theorem]{Corollary}
\newtheorem{remark}[theorem]{Remark}

\theoremstyle{definition}
\newtheorem{definition}[theorem]{Definition}
\newtheorem{example}[theorem]{Example}

\def\bC{{\mathbb C}}
\def\bG{{\mathbb G}}
\def\bL{{\mathbb L}}
\def\bN{{\mathbb N}}
\def\bP{{\mathbb P}}
\def\bQ{{\mathbb Q}}
\def\bR{{\mathbb R}}
\def\bZ{{\mathbb Z}}
\def\sbR{\mathsmaller{\bR}}
\def\sbZ{\mathsmaller{\bZ}}
\def\mT{\mathsmaller{\mathsf{T}}}
\def\ka{{\mathcal A}}

\def\ke{{\mathcal E}}
\def\kf{{\mathcal F}}
\def\ko{{\mathcal O}}

\def\ks{{\mathcal S}}
\def\kz{{\mathcal Z}}
\def\ba{{\boldsymbol{a}}}
\def\bb{{\boldsymbol{b}}}
\def\bc{{\boldsymbol{c}}}
\def\bd{{\boldsymbol{d}}}
\def\bj{{\boldsymbol{j}}}
\def\bm{{\boldsymbol{m}}}
\def\bq{{\boldsymbol{q}}}
\def\br{{\boldsymbol{r}}}
\def\bt{{\boldsymbol{t}}}
\def\bx{{\boldsymbol{x}}}
\def\by{{\boldsymbol{y}}}
\def\bv{{\boldsymbol{v}}}
\def\bz{{\boldsymbol{z}}}
\def\bfE{{\mathbf E}}
\def\bfF{{\mathbf F}}
\def\bfG{{\mathbf G}}
\def\bpi{{\boldsymbol{\pi}}}
\def\bzero{{\boldsymbol{0}}}
\def\bone{{\boldsymbol{1}}}
\def\codim{\operatorname{codim}}
\def\cone{\operatorname{cone}}
\def\conv{\operatorname{conv}}
\def\eu{\operatorname{Eu}}
\def\im{\operatorname{im}}
\def\pr{\operatorname{pr}}
\def\rank{\operatorname{rank}}
\def\RSV{\operatorname{RSV}}
\def\Sing{\operatorname{Sing}}
\def\vol{\operatorname{vol}}
\def\Vol{\operatorname{Vol}}
\newcommand{\ma}{{\mathscr A}}
\newcommand{\mb}{{\mathscr B}}
\newcommand{\Pd}{{\mathbb{P}}^{\boldsymbol{d}}}
\def\Salpha{\ks_\alpha}

\makeatletter
\renewcommand*\env@matrix[1][*\c@MaxMatrixCols c]{%
    \hskip -\arraycolsep
    \let\@ifnextchar\new@ifnextchar
    \array{#1}}
\makeatother

\title{Nashtopes and their A-discriminants}

\author[E. Neuhaus]{Elke Neuhaus}
\author[I. Portakal]{Irem Portakal}
\address{Max Planck Institute for Mathematics in the Sciences, Leipzig, Germany}
\email{elke.neuhaus@mis.mpg.de}
\email{mail@irem-portakal.de}

\author[L. Sodomaco]{Luca Sodomaco}
\address{Max Planck Institute for Mathematics in the Sciences, Leipzig, Germany}
\curraddr{Institut f\"{u}r Mathematik, Universit\"{a}t Osnabr\"{u}ck, Germany}
\email{luca.sodomaco@mis.mpg.de, luca.sodomaco@uni-osnabrueck.de}

\author[K. Tung]{Katherine Tung}
\address{Cambridge University, Department of Pure Mathematics and Mathematical Statistics, Cambridge, UK}
\email{kat47@cam.ac.uk}

\subjclass[2020]{52B20, 52B11, 05E14, 14A10, 91A06, 91A12, 91A80, 14M25}

\keywords{}

\begin{document}

\begin{abstract}
    An $n$-player game in normal form is said to be of format $\bd = (d_1, \dots, d_n)$ if the $i$-th player has $d_i$ options for $1 \le i \le n$, and is said to be balanced if $d_n - 1 \le \sum_{i=1}^{n-1}(d_i - 1)$. The Nash equilibria of the game correspond to solutions of the system of indifference equations for the game, with payoff differences as coefficients and strategy probabilities as variables. As such, the mixed discriminant of this system gives some information about which games could have an unusual number of totally mixed Nash equilibria. This mixed discriminant can be regarded as an $\ka$-discriminant for some matrix $\ka$ in the sense of [Gelfand--Kapranov--Zelevinsky 1994], and the convex hull of the columns of $\ka$ is a polytope, which we call the \emph{Nashtope}.

    We study this $\ka$-discriminant for each balanced game format with at least three players. This discriminant is a subset of the Nash discriminant variety introduced in [Abo--Portakal--Sodomaco 2026]. We compute the codimension and degree of this discriminant for all balanced game formats, including on-boundary formats, using the Euler obstruction formula in [Matsui--Takeuchi 2011]. Our $\ka$-discriminant locus corresponds to a special multigraded Hurwitz locus studied in [Pratt--Sodomaco--Sturmfels 2026]. As a consequence of [Pokidkin 2026], we obtain the codimensions and degrees of the irreducible components of the Nash discriminant variety for each balanced game format with at least three players, fully resolving Problem 3.8 of [Abo--Portakal--Sodomaco 2026].
\end{abstract}

\maketitle

\tableofcontents

\section{Introduction}
The set of Nash equilibria \cite{nash1950equilibrium} of a finite game in normal form of format $\bd \coloneqq (d_1,\dots, d_n)$ is a semi-algebraic set. Throughout, we assume $d_1 \leq \cdots \leq d_n$. For \emph{totally mixed Nash equilibria}, where every strategy is played with positive probability, this semi-algebraic set is given by multilinear equations together with the linear conditions of the product of open simplices. If one neglects the simplices and instead considers only the multilinear equations, one obtains a system that is generically reduced and zero-dimensional. This system defines the \emph{Nash equilibrium scheme} of the game, which can also be described as the zero scheme of a global section of a vector bundle that is a direct sum of line bundles of multidegree given by the aforementioned multilinear generators~\cite[Section 2]{abo2026vector}. The generic picture is classical: by the Bernstein-Khovanskii-Kushnirenko theorem, the number of totally mixed equilibria is constant on a Zariski open set of games and equals the mixed-volume count of McKelvey and McLennan~\cite{mckelvey1997maximal}.

Much less is understood about \emph{nongeneric} games, for which these generic results do not generally apply. Abo, Portakal and Sodomaco introduced the \emph{Nash discriminant and Nash resultant varieties}: the locus of games whose Nash equilibrium scheme is nonreduced or has a positive-dimensional component~\cite[Definition 3.2, Definition 3.29]{abo2026vector}. For two players, the generators of the Nash equilibrium scheme are linear, and the nongeneric behavior of the solution system is described by determinantal varieties~\cite[Proposition~3.3, Proposition~3.30]{abo2026vector}, so this case is well-understood. For $n \geq 3$, the behavior depends sharply on the format. A format $\bd$ is called \textit{balanced} if $d_n - 1 \le \sum_{i=1}^{n-1}(d_i - 1)$, and \textit{unbalanced} otherwise. In unbalanced formats, a generic game has no totally mixed equilibria at all; this falls under the case of the Nash resultant variety, which is well-characterized with known dimension and degree \cite[Theorem 3.31]{abo2026vector}. On the boundary between the balanced and unbalanced cases, the Nash discriminant variety is reducible \cite[Theorem 3.26]{abo2026vector}; for balanced formats, the Nash discriminant variety contains a hypersurface component of real codimension one in the space of games \cite[Theorem 3.7]{abo2026vector}. In the last two cases, however, the degree and dimension of the Nash discriminant variety were not completely known \cite[Problem 3.8]{abo2026vector}. This paper fully resolves that open question.

A fundamental obstruction to classifying Nash equilibrium schemes comes from Datta's universality theorem~\cite{datta2003universality}: every real algebraic variety arises as the Nash equilibrium scheme of a certain three-player game or a binary game. Universality thus rules out a structural classification of Nash equilibrium schemes in general, but it is silent on a complementary question: \emph{where and how}, within the space of games of a fixed format, do degenerate equilibrium behaviors live? Determining the degree of the Nash discriminant variety is the first step toward understanding this question. The $\ka$-discriminant variety is an inherently complex-algebraic object, capturing degeneracies in the scheme-theoretic sense; the relationship with the semialgebraic set of actual totally mixed Nash equilibria inside the open simplex requires separate techniques and is not treated here. The degree formula provides a first quantitative geography of the degenerate locus. Its dependence on the format calibrates how the nongeneric behavior that Datta's universality theorem guarantees distributes across game space as the format grows.

Given a game format $\bd$, we define the \emph{Nashmatrix} $\ka$ as the Cayley matrix for the exponent vectors of monomials in the equations defining the {Nash equilibrium scheme}. Furthermore, the \emph{Nashtope} of format $\bd$ is the convex hull of the columns of $\ka$ (see Definition~\ref{def: Nashmatrix} and \ref{def: Nashtope}). 
In Section~\ref{sec: degree of the A-discriminant}, we study the $\ka$-discriminant $\nabla_{\!\ka}$ and give its dimension and degree for all balanced game formats for which the roadmap was established in \cite{matsui2011geometric}. Along the way, we give a combinatorial characterization of the full face structure of the Nashtope via so-called \emph{face strings} (Theorem~\ref{thm:general-bijection}). 

Furthermore, using results of \cite{pokidkin}, we recognize $\nabla_{\!\ka}$ as an irreducible component of the Nash discriminant variety defined in \cite[Definition 3.2]{abo2026vector}. We prove that the Nash discriminant variety is irreducible for strictly balanced game formats and therefore that it equals $\nabla_{\!\ka}$. In the on-boundary case, we deduce that the Nash discriminant variety has two irreducible components, one of codimension one and one of codimension two. The first has degree prescribed by \cite[Theorem~3.26]{abo2026vector}, while the second equals $\nabla_{\!\ka}$ and thus has degree given by that of $\nabla_{\!\ka}$, which is extensively studied in Section~\ref{sec: degree of the A-discriminant}. These results, which are stated in detail below, are for each balanced game format $\bd$ with at least three players. 

\begin{theorem*}[Strictly balanced game formats, Corollary \ref{cor:final-discriminant-strictly-balanced}]
For strictly balanced game formats $\bd$, the Nash discriminant variety $\Delta(\bd)$ is an irreducible hypersurface. Its degree is the coefficient of $\prod_{i=1}^n z_i^{d_i-1}$ in 
\[
    \left(C + \sum_{i=1}^n(C - 2d_i + 1)z_i\right) \prod_{i=1}^n(1 + \hat{z}_i)^{d_i-1}\,
\]
where $\hat{z}_i\coloneqq\sum_{j\neq i}z_j$ and $C := \sum_{i=1}^n (d_i - 1)$.
\end{theorem*}

\begin{theorem*}[On-boundary game formats, Corollary \ref{cor:discriminant-on-boundary}]
For on-boundary game formats $\bd$, the Nash discriminant variety $\Delta(\bd)$ has two irreducible components, one of codimension 1 and one of codimension 2. The first has degree
\[
    \frac{(d_n - 1)!}{(d_1 - 1)! \cdots (d_{n-1} - 1)!} \left(d_n - 1 + \sum_{i=1}^{n-1} (d_i - 1)(d_n - d_i - 1)\right)
\]
and the second has degree
\[
    \frac{1}{2} \binom{d_n - 1}{d_1 - 1,\dots,d_{n-1} - 1}\left(d_n(d_n - 1) + (4d_n - 1)e_2(\bd-\bone) + 2 e_2(\bd-\bone)^2\right)
\]
where $e_2(\bd-\bone) = \sum_{1 \le i < j \le n - 1}(d_i - 1)(d_j - 1)$.
\end{theorem*}

Section~\ref{sec: polynodality} is independent of the rest of the paper but attacks the same question. Here, we interpret the multilinear equations of totally mixed Nash equilibria as linear equations after composing them with the mixed Segre embedding of Equation \eqref{eq: mixed Segre embedding}. On the one hand, our $\ka$-discriminant variety is a \emph{multigraded Hurwitz locus} in the spirit of \cite{pratt2026multigraded}. On the other hand, the Nash discriminant variety $\Delta(\bd)$ is a \emph{(multiprojective) higher associated variety} in the sense of \cite[Section 4.5]{dogan2023complexity}. For strictly balanced game formats, \cite[Proposition 5.5]{pratt2026multigraded} gives an upper bound on the multidegree of $\nabla_{\!\ka}$. In Corollary~\ref{cor: A-discriminant and multigraded Hurwitz form}, we show that this upper bound is attained by showing that the mixed Segre embedding of the \emph{multiprojective strategy space} $\Pd=\prod_{i=1}^n\bP^{d_i-1}$ satisfies a geometric property called \emph{polynodality}, introduced in \cite{pratt2026multigraded}.

As a supplement to this work, we have implemented \texttt{Macaulay2} and \texttt{Julia} software for generating finite normal-form games on the $\ka$-discriminant variety for a given input Nashmatrix~$\ka$. Instead of computing the defining ideal of the $\ka$-discriminant
variety, which is computationaly expensive, the code samples directly from its Horn-Kapranov uniformization. We describe the main functions in Examples~\ref{ex: 2222} and \ref{ex: 223 part 2}. The code and documentation are available in the Zenodo repository \cite{Zenodo}.

\section{Preliminaries}\label{sec: prelim}

We adopt the notation introduced in \cite{abo2026vector}. For a positive integer $n$, we let $[n] := \{1, \dots, n\}$. Consider an $n$-player normal-form game $X$ of format $\bd = (d_1, \dots, d_n)$ with $d_1 \leq \cdots \leq d_n$. It is defined by $n$ payoff tensors $X^{(1)}, \ldots, X^{(n)}$ of format $d_1 \times \cdots \times d_n$, indexed by the joint strategies of the players. The space of joint strategies is $I \coloneqq \prod_{i=1}^n [d_i]$ and we furthermore set $I_{-i} \coloneqq \prod_{j \in [n] \setminus \{i\}} [d_j]$. For a joint strategy $\bj = (j_1, \ldots, j_n) \in I$, we write
\[
    \bj_{-i} \coloneqq (j_1, \ldots, j_{i-1}, j_{i+1}, \ldots, j_n)
\]
for the element of $I_{-i}$ obtained by removing $j_i$ from $\bj$ and
\[
    (k,\bj_{-i}) \coloneqq (j_1, \ldots, j_{i-1}, k, j_{i+1}, \ldots, j_n)
\]
for the element of $I$ obtained from $\bj_{-i}$ by inserting $k \in [d_i]$ in its $i$-th position. 
The actual strategies of each player $i$ are denoted by $\pi^{(i)}=(\pi^{(i)}_1, \ldots , \pi^{(i)}_{d_i})^\mT$ and their collection as $\bpi \coloneqq (\pi^{(1)}, \ldots, \pi^{(n)})$. Given $j = (j_1, \ldots, j_n) \in I$, we further set 
\[
    \pi_{\bj_{-i}} \coloneqq \prod_{k \in [n] \setminus \{i\}} \pi^{(k)}_{j_k}\,.
\]
Denote the open probability simplex of dimension $d_i -1$ by $\Delta^{\circ}_{d_i -1}$ and the entries of the payoff tensor $X^{(i)}$ by $x^{(i)}_\bj$. Then, the set of totally mixed Nash equilibria consists of the points $\bpi \in \prod_{i=1}^n \Delta^{\circ}_{d_i-1}$ that cause the polynomials
\begin{equation}\label{eq: multihomog gens}
    \Delta f^{(i)}_{1,k} \coloneqq \sum_{\bj_{-i} \in I_{-i}} (x^{(i)}_{(1,\bj_{-i})} - x^{(i)}_{(k, \bj_{-i})}) \pi_{\bj_{-i}}
\end{equation}
for $i\in[n]$ and $k\in[d_i]\setminus\{1\}$ to vanish.
This gives rise to the \emph{Nash equilibrium scheme} of the game $X$, the subscheme $\kz_X$ of the multiprojective strategy space $\bP^{\bd} \coloneqq \prod_{i=1}^n \bP^{d_i-1}$ defined by the multihomogeneous ideal $\langle \Delta f^{(i)}_{1,k} \mid i \in [n], k \in [d_i]\setminus\{1\}\rangle$.

\begin{definition}
An $n$-player game format $\bd$ is \emph{balanced} if $d_n - 1 \leq \sum_{i=1}^{n-1} (d_i - 1)$ and \emph{unbalanced} if $d_n - 1 > \sum_{i=1}^{n-1} (d_i - 1)$.
A balanced game of format $\bd$ is \emph{strictly balanced} if $d_n - 1 < \sum_{i=1}^{n-1} (d_i - 1)$, and \emph{on-boundary} if $d_n - 1 = \sum_{i=1}^{n-1} (d_i - 1)$.
\end{definition}

If a game is unbalanced, its Nash equilibrium scheme is generically empty \cite[Theorem 2.8]{abo2026vector}. Otherwise, it is generically zero-dimensional of degree given by the mixed volume of the Newton polytopes as in \cite{mckelvey1997maximal}. The non-genericity for balanced games of a certain format $\bd$ is captured in the \emph{Nash discriminant variety} $\Delta(\bd)$; see \cite[Definition 3.2]{abo2026vector}. It is the space of games of format $\bd$ whose Nash discriminant scheme $\kz_X$ is non-reduced or positive-dimensional.

We now formalize the definition of the Nashtope and the Nashmatrix $\ka$. 

\begin{definition}[adapted from \cite{cattani2013mixed}]
The \emph{Cayley matrix} for the point configurations $A_1, \dots, A_m \subseteq \bZ^\ell$ is defined as 
\begin{equation}\label{eq: cayley config}
    \operatorname{Cay}(A_1,\ldots,A_m) =
    \begin{pmatrix}
        \bone & \bzero & \cdots & \bzero \\
        \bzero & \bone & \cdots & \bzero \\
        \vdots & \vdots & \ddots & \vdots \\
        \bzero & \bzero & \cdots & \bone \\
        A_1 & A_2 & \cdots & A_m
    \end{pmatrix}
\end{equation}
The matrix has $m + \ell$ rows and $\sum_{i=1}^m |A_i|$ columns; hence  $\bzero = (0, \ldots, 0)$ and $\bone = (1, \ldots, 1)$ represent row vectors of suitable dimensions. 
\end{definition}

The Nashmatrix $\ka$ of format $\bd$ is a special type of Cayley matrix.
Consider the $\sum_{i=1}^n(d_i-1)$ generators $\Delta f^{(i)}_{1,k}$ from \eqref{eq: multihomog gens} of the Nash equilibrium scheme.
Order the probability variables $\pi_j^{(i)}$ lexicographically by player number $i$ and then by option number $j$. Furthermore, we order $(i,k)$-lexicographically the set of equations $\Delta f^{(i)}_{1,k}$. For every $i\in[n]$, we let $A_i$ be the matrix obtained by stacking the column exponent vectors of all monomials in the equations $\Delta f^{(i)}_{1,k}$.

\begin{definition}\label{def: Nashmatrix}
The \emph{Nashmatrix} of format $\bd$ is the matrix with $\sum_{i=1}^n (d_i-1) + \sum_{i=1}^n d_i$ rows and $\sum_{i=1}^n(d_i-1)\prod_{j\neq i}d_j$ columns
\[
    \ka\coloneqq\operatorname{Cay}(A_1,\ldots,A_C)\,.
\]
Here, $C \coloneqq \sum_{i=1}^n (d_i - 1)$ is the number of generators of the multihomogeneous ideal \eqref{eq: multihomog gens}.
\end{definition}

The Nashmatrix can be defined for a game of any format $\bd$, but since our motivation is the study of the $\ka$-discriminant as a subset of the Nash discriminant variety, we only consider the Nashmatrices for balanced game formats from here on.

\begin{example}[Binary 3-player games]
Let $X$ be a game of format $\bd=(2,2,2)$.
Define
\[
    a_{jk}^{(1)}\coloneqq x_{2jk}^{(1)} - x_{1jk}^{(1)}\,,\quad a_{ik}^{(2)}\coloneqq x_{i2k}^{(2)} - x_{i1k}^{(2)}\,,\quad a_{ij}^{(3)}\coloneqq x_{ij2}^{(3)} - x_{ij1}^{(3)}\quad\forall\,i,j,k\in[2]\,.
\]
Then the system of equations $\Delta f^{(i)}_{1,k}$ is
\begin{equation}\label{eq: 222system}
\begin{cases}
    \Delta f_{1,2}^{(1)} = a_{11}^{(1)}\,\pi_1^{(2)}\pi_1^{(3)} + a_{12}^{(1)}\,\pi_1^{(2)}\pi_2^{(3)}+ a_{21}^{(1)}\,\pi_2^{(2)}\pi_1^{(3)} + a_{22}^{(1)}\,\pi_2^{(2)}\pi_2^{(3)} = 0\\[2pt]
    \Delta f_{1,2}^{(2)} = a_{11}^{(2)}\,\pi_1^{(1)}\pi_1^{(3)} + a_{12}^{(2)}\,\pi_1^{(1)} \pi_2^{(3)} + a_{21}^{(2)}\,\pi_2^{(1)}\pi_1^{(3)} + a_{22}^{(2)}\,\pi_2^{(1)}\pi_2^{(3)} = 0\\[2pt]
    \Delta f_{1,2}^{(3)} = a_{11}^{(3)}\,\pi_1^{(1)}\pi_1^{(2)} + a_{12}^{(3)}\,\pi_1^{(1)}\pi_2^{(2)} + a_{21}^{(3)}\,\pi_2^{(1)} \pi_1^{(2)} + a_{22}^{(3)}\,\pi_2^{(1)}\pi_2^{(2)} = 0\,.
\end{cases}
\end{equation}
For all $i\in[3]$, let $A_i$ be the set of exponent vectors of monomials in the $i$-th equation of system \eqref{eq: 222system}, where we order the variables as  $\pi_1^{(1)}, \pi_2^{(1)}, \pi_1^{(2)}, \pi_2^{(2)}, \pi_1^{(3)}, \pi_2^{(3)}$. We obtain the $9 \times 12$ Nashmatrix
\vspace{0.4cm}
\begin{equation}\label{eq: A matrix for 222}
\ka = \operatorname{Cay}(A_1,A_2,A_3) =
\begingroup
\renewcommand*{\arraystretch}{1.1}
\begin{pNiceArray}{cccc|cccc|cccc}[first-row,last-col,margin]
    \Hdotsfor{4}[line-style=mybrace, shorten=0pt]^{\Delta f_{1,2}^{(1)}} & \Hdotsfor{4}[line-style=mybrace, shorten=0pt]^{\Delta f_{1,2}^{(2)}} & \Hdotsfor{4}[line-style=mybrace, shorten=0pt]^{\Delta f_{1,2}^{(3)}} & \\
    1&1&1&1&0&0&0&0&0&0&0&0& \\
    0&0&0&0&1&1&1&1&0&0&0&0& \\
    0&0&0&0&0&0&0&0&1&1&1&1& \\\hline
    0&0&0&0&1&1&0&0&1&1&0&0& \pi_1^{(1)}\\
    0&0&0&0&0&0&1&1&0&0&1&1& \pi_2^{(1)}\\
    1&1&0&0&0&0&0&0&1&0&1&0& \pi_1^{(2)}\\
    0&0&1&1&0&0&0&0&0&1&0&1& \pi_2^{(2)}\\
    1&0&1&0&1&0&1&0&0&0&0&0& \pi_1^{(3)}\\
    0&1&0&1&0&1&0&1&0&0&0&0& \pi_2^{(3)}
\end{pNiceArray}
\endgroup.
\end{equation}
\end{example}

\begin{definition}[Nashtope]\label{def: Nashtope}
The convex hull $P$ of the columns of $\ka$ corresponding to a game format $\bd$ is called its \emph{Nashtope}. Note that each column is a vertex of the Nashtope as opposed to an interior point.
A face of a Nashtope is called a \emph{Nashface}. Note that the order of the columns of $\ka$ does not affect the definition of the Nashtope $P$. Despite this, we have chosen a fixed order to make the definition of Nashmatrix $\ka$ concrete.
\end{definition}

\begin{definition}[$\ka$-discriminant variety $\nabla_{\!\ka}$]
Let $r$ be the number of rows and $s$ be the number of columns $a_i$ of $\ka$. Write $\theta = (\theta_1,\dots,\theta_r)$ for the coordinate variables on the algebraic torus $(\bC^*)^r$, so that $\theta^{a_i} = \theta_1^{a_{i1}}\cdots\theta_r^{a_{ir}}$. For any fixed $z=(z_1,\ldots,z_s)\in\bC^s$, the polynomial $f_{\ka,z}(\theta)\coloneqq z_1\,\theta^{a_1}+\cdots+z_s\,\theta^{a_s}$ defines a subvariety $H_{\ka,z}\subseteq(\bC^*)^r$. The variety $H_{\ka,z}$ is \emph{singular} if there exists $\theta\in(\bC^*)^r$ such that $f_{\ka,z}(\theta)=\frac{\partial f_{\ka,z}}{\partial\theta_1}(\theta)=\cdots=\frac{\partial f_{\ka,z}}{\partial\theta_r}(\theta)=0$. 
The \emph{$\ka$-discriminant variety} of the Nashtope in the sense of \cite[Chapter 9]{GKZ94} is the Zariski closure
\[
    \nabla_{\!\ka} \coloneqq \overline{\left\{ z \in \bP^{s-1} \mid \text{$H_{\ka,z}$ is singular}\right\}}\,.
\]
\end{definition}

By definition, the $\ka$-discriminant variety $\nabla_{\!\ka}$ is a subset of the Nash discriminant variety $\Delta(\bd)$. \emph{A priori}, it is unclear whether $\nabla_{\!\ka}$ is an irreducible component of $\Delta(\bd)$ or, in some cases, the whole variety. The first step is to study the $\ka$-discriminant variety more closely, namely, by computing its dimension and degree. 

Throughout the paper, we adopt the following notations.
Let $\#\ka$ be the number of columns of $\ka$. The expression ``$a \in \ka$'' means that $a$ is a column of $\ka$. We usually denote by $F$ a nonempty face of the Nashtope $P$ (denoted $F \le P$). We also let $\ka_F$ be the vertex set of $F$. We conclude the preliminaries with a concrete example which motivates the study of $\ka$-discriminants of Nashmatrices.

\begin{example}\label{ex: 2222}
Consider the game format $\bd=(2,2,2,2)$. The Nashmatrix $\ka$ has format $r\times s=12\times 32$. By Corollary \ref{cor:final-discriminant-strictly-balanced}, the $\ka$-discriminant variety $\nabla_{\!\ka}\subseteq\bP^{31}$ is a hypersurface of degree 80. Its defining polynomial involves 32 unknowns, with the coefficients (for $i,j,k,\ell\in[2]$)
\[
    a_{jk\ell}^{(1)} \coloneqq X_{2jk\ell}^{(1)}-X_{1jk\ell}^{(1)}\,\quad a_{ik\ell}^{(2)} \coloneqq X_{i2k\ell}^{(2)}-X_{i1k\ell}^{(2)}\,\quad a_{ij\ell}^{(3)} \coloneqq X_{ij2\ell}^{(3)}-X_{ij1\ell}^{(3)}\,\quad a_{ijk}^{(4)} \coloneqq X_{ijk2}^{(4)}-X_{ijk1}^{(4)}
\]
of the polynomials defining the multihomogeneous ideal $\langle\Delta f_{1,2}^{(1)},\ldots,\Delta f_{1,2}^{(4)}\rangle$. Hence computing the $\ka$-discriminant is rather challenging. In addition, one should compare this $\ka$-discriminant with the {\em hyperdeterminant} of format $2\times 2\times 2\times 2$. This is another $\ka'$-discriminant, where $\ka'$ is the matrix of the $4$-cube. See \cite{HSYY08} for a more in-depth study of the monomials appearing in this hyperdeterminant. However, one can easily sample a generic point on $\nabla_{\!\ka}$ via the so-called \emph{Horn-Kapranov uniformization}
\[
    (\bC^*)^r\times\bP^{s-\rank(\hat{\ka})-1} \to (\bP^{s-1})^*\,,\quad (t,u)\mapsto[t^{-a_1}(b_1\cdot u):\cdots:t^{-a_s}(b_s\cdot u)]\,,
\]
where $t=(t_1,\ldots,t_r)$, $u=[u_1:\cdots:u_{s-\rank(\hat{\ka})-1}]$, $a_1,\ldots,a_s$ are the columns of $\ka$. Here $\hat{\ka}$ is a matrix obtained by adding a row of ones to $\ka$, and $b_1,\ldots,b_s$ are the rows of $B\in\bZ^{s\times(n-\rank(\hat{\ka}))}$, which is a right-kernel matrix of $\hat{\ka}$. We refer to \cite[Chapter 9.3.C]{GKZ94} and \cite[Section 7.3]{telentoricbook} for more details.
In the supplementary material \cite{Zenodo}, we developed functions in \texttt{Macaulay2} \cite{grayson1997macaulay2} and \texttt{Julia} \cite{julia} to sample a sufficiently random point on the $\ka$-discriminant and its related games.
For example, using the \texttt{Macaulay2} function \texttt{sampleOnAdiscr} of \texttt{randomGameOnADiscriminant.m2}, we sampled the point $[a_{111}^{(1)}:\cdots:a_{222}^{(4)}]\in\bP^{31}$ with the 32 coordinates displayed in Table~\ref{table: coefficients format 2x2x2x2 on A-discriminant}.
This point can be realized by the game $X=(X^{(1)},\ldots,X^{(4)})$ with the payoff tensors displayed in Table~\ref{table: payoff tensors 2x2x2x2 game on A-discriminant} and by \cite[Remark 3.4]{abo2026vector} the algebraic invariants of the Nash discriminant are unchanged in passing from the space of games to the coefficient space $\bP^{31}$.
Using the \texttt{Macaulay2} package \texttt{GameTheory} \cite{GameTheorySource,connelly2025gametheory}, we verified that the Nash equilibrium scheme $\kz_X\subseteq(\bP^1)^4$ is zero-dimensional, nonreduced of degree 9, and supported at 8 points, implying that one of these points has multiplicity two. The latter is precisely $([1:1],[1:1],[1:1],[1:1])$, or equivalently the totally mixed Nash equilibrium
\begin{equation}\label{eq: coordinates double tmNE 2222}
    \bpi = \left(\left(\frac{1}{2},\frac{1}{2}\right),\left(\frac{1}{2},\frac{1}{2}\right),\left(\frac{1}{2},\frac{1}{2}\right),\left(\frac{1}{2},\frac{1}{2}\right)\right)\in(\Delta_1^\circ)^4\,.
\end{equation}
Furthermore, with the package \texttt{RealRoots} \cite{RealRootsSource,lopez2024real}, we checked that $\kz_X$ contains three additional real points, one of which corresponds to the totally mixed Nash equilibrium (rounded to five decimal digits)
\[
    \left((0.35751,0.64249),(0.66634,0.33366),(0.77250,0.22750),(0.68369,0.31631)\right)\in(\Delta_1^\circ)^4\,.
\]
The code used is saved in the file \texttt{code\_example\_2.6.m2}; see \cite{Zenodo}.
\end{example}

\begin{table}[htbp]
\centering
\begin{tabular}{crrrr}
\toprule
 & $a_{jk\ell}^{(1)}$ & $a_{ik\ell}^{(2)}$ & $a_{ij\ell}^{(3)}$ & $a_{ijk}^{(4)}$ \\
\midrule
111 & $ -1$ & $ 1$ & $ 3$ & $ 7$ \\
112 & $-10$ & $-1$ & $-3$ & $-6$ \\
121 & $ 14$ & $ 6$ & $ 6$ & $-1$ \\
122 & $ -1$ & $ 0$ & $-2$ & $ 2$ \\
211 & $  8$ & $-3$ & $-2$ & $-1$ \\
212 & $ -4$ & $ 7$ & $-4$ & $ 8$ \\
221 & $-13$ & $-1$ & $ 4$ & $-8$ \\
222 & $  7$ & $-9$ & $-2$ & $-1$ \\
\bottomrule
\end{tabular}
\caption{A point on the $\ka$-discriminant variety of format $\bd=(2,2,2,2)$.}
\label{table: coefficients format 2x2x2x2 on A-discriminant}
\end{table}

\begin{table}[htbp]
\centering
\begin{tabular}{crrrr}
\toprule
$(j_1 j_2 j_3 j_4)$ & $X^{(1)}$ & $X^{(2)}$ & $X^{(3)}$ & $X^{(4)}$ \\
\midrule
1111 & $  1$ & $ 1$ & $-1$ & $ 3$ \\
1112 & $  0$ & $ 1$ & $-2$ & $10$ \\
1121 & $  3$ & $ 2$ & $ 2$ & $ 0$ \\
1122 & $ -2$ & $-3$ & $-5$ & $-6$ \\
1211 & $  1$ & $ 2$ & $-3$ & $ 0$ \\
1212 & $  1$ & $ 0$ & $ 3$ & $-1$ \\
1221 & $ -1$ & $ 8$ & $ 3$ & $ 1$ \\
1222 & $  2$ & $-3$ & $ 1$ & $ 3$ \\
2111 & $  0$ & $ 0$ & $ 3$ & $ 1$ \\
2112 & $-10$ & $-2$ & $-2$ & $ 0$ \\
2121 & $ 17$ & $-3$ & $ 1$ & $ 2$ \\
2122 & $ -3$ & $ 1$ & $-6$ & $10$ \\
2211 & $  9$ & $-3$ & $-3$ & $-1$ \\
2212 & $ -3$ & $ 5$ & $-3$ & $-9$ \\
2221 & $-14$ & $-4$ & $ 1$ & $-3$ \\
2222 & $  9$ & $-8$ & $-5$ & $-4$ \\
\bottomrule
\end{tabular}
\caption{Payoff tensors of a four-player binary game corresponding to a point on the $\ka$-discriminant.}
\label{table: payoff tensors 2x2x2x2 game on A-discriminant}
\end{table}

\begin{remark}
The \texttt{Macaulay2} function \texttt{sampleOnAdiscr} of \texttt{randomGameOnADiscriminant.m2} samples a random point $z\in\bP^{s-1}$ on the $\ka$-discriminant variety with coordinates in $\bZ$. Randomness ensures that the variety $H_{\ka,z}$ is singular at only one point $\theta\in(\bC^*)^r$. By \cite[Proposition 7.1.11]{telentoricbook}, we can write the unique singular point with coordinates in $\bZ$. Translated into payoff tensors, our code provides a sufficiently random game $X$ of balanced format $\bd$ with integer payoff tensors whose Nash equilibrium scheme $\kz_X\subseteq\bP^\bd$ has a unique double point representable with coordinates in $\bZ$. If these coordinates are all positive, the point corresponds to a totally mixed Nash equilibrium with coordinates in $\bQ$ (via normalization), as in Example~\ref{ex: 2222} with the point $\bpi$ in \eqref{eq: coordinates double tmNE 2222}.
\end{remark}

\section{The degree of the \texorpdfstring{$\ka$}{A}-discriminant for balanced game formats}\label{sec: degree of the A-discriminant}

Let us assume that the game format $\bd$ is balanced. Our aim in this section is to compute the degree of the $\ka$-discriminant variety $\nabla_{\!\ka}$. As a byproduct, we also obtain its dimension and the combinatorial structure of the Nashfaces. We will show that for strictly balanced game formats, $\nabla_{\!\ka}$ has codimension $1$, while for on-boundary games it has codimension $2$.
For the rest of the paper, we assume that $n \ge 3$. The case $\bd=(d,d)$ is handled already in \cite[Proposition 3.3]{abo2026vector}, where the Nash discriminant variety has two irreducible components of codimension two and degree $\binom{d}{2}$. 

\subsection{Roadmap for the degree calculation}

To calculate $\deg(\nabla_{\!\ka})$, we use the following result. The original statement expresses the dimension and degree of general dual toric varieties in terms of the volume and Euler obstruction of the faces of the associated polytope. Recall that $P$ is the Nashtope which is the convex hull of the columns of the Nashmatrix $\ka$.

\begin{theorem}[{\cite[Theorem 1.4]{matsui2011geometric}}]\label{thm:polytope}
For $1 \le i \le \# \ka - 1$, set
\[
    \delta_i \coloneqq \sum_{F \le P} (-1)^{\codim(F)} \left\{ \binom{\dim(F) - 1}{i} + (-1)^{i-1}(i + 1)\right\} \Vol_\sbZ(F)\eu(F)\,.
\]
Then $r \coloneqq \codim(\nabla_{\!\ka}) = \min\{i \mid \delta_i \neq 0\}$ and $\deg(\nabla_{\!\ka}) = \delta_r$.
\end{theorem}

In general, the binomial coefficient used in Theorem~\ref{thm:polytope} is the generalized binomial coefficient $\binom{p}{q}\coloneqq \frac{p(p - 1)(p - 2) \cdots (p - q + 1)}{q!}$ for any integer $p$ and nonnegative integer $q$. For example, when $\dim(F) = 0$, we have $\binom{\dim(F) - 1}{i} = \binom{-1}{i} = (-1)^i$. In particular, in order to show that $\nabla_{\!\ka}$ is a hypersurface for strictly balanced game formats, it suffices to verify that
\begin{equation}\label{eq: delta1}
    0 \neq \delta_1 = \sum_{F \le P} (-1)^{\codim(F)}(\dim(F) + 1)\Vol_\sbZ(F)\eu(F)\,.
\end{equation}
Since the above expression coincides with \cite[Chapter 9, Theorem 2.8]{GKZ94} when $\eu(F)=1$ for every face of $P$, we may view \cite[Theorem 1.4]{matsui2011geometric} as a generalization of it.

\noindent For on-boundary game formats, i.e. $d_n-1=\sum_{i=1}^{n-1}(d_i-1)$, we will instead focus on
\begin{equation}\label{eq: delta2}
    \delta_2 = \sum_{F \le P} (-1)^{\codim(F)}\left(\binom{\dim(F) - 1}{2} - 3\right)\Vol_\sbZ(F)\eu(F)\,.
\end{equation}

In the rest of this section, we unwind the definitions of $\Vol_\sbZ$ and $\eu$, including the auxiliary functions $u$ and $i$ given in \cite[Section 4.2]{matsui2011geometric}. Critically, the Euler obstruction $\eu(F)$ depends not just on the polytope $P$ with $F$ as a face, but also on the lattice generated by the columns of~$\ka$.

Choose $M$ to be a $\bZ$-lattice (free $\bZ$-module) of rank $r(M)$ so that the rank of $M(\ka)$, the affine $\bZ$-lattice generated by $\ka$, is $r(M)$. We let $P$ be the convex hull of the vertices given by the columns of $\ka$ in $M_\sbR \coloneqq \bR \otimes_\sbZ M$. For each face $F_\alpha$ of $P$, we consider the smallest affine subspace $\bL(F_\alpha)$ of $M_\sbR$ containing $F_\alpha$ and the affine $\bZ$-lattice $M_\alpha \coloneqq M(\ka \cap F_\alpha)$ generated by $\ka \cap F_\alpha$ in $\bL(F_\alpha)$.
Observe that $M(\ka \cap F)$ is a $\bZ$-lattice of rank $\dim(F) = \dim(\bL(F))$ in $\bL(F)$ and we have $(M(\ka \cap F))_\sbR \cong \bL(F)$. Let $\vol$ be the Lebesgue measure of $(\bL(F), M(\ka \cap F))$ by which the volume of the fundamental domain of the translation by $M(\ka \cap F)$ on $\bL(F)$ is measured to be~$1$. For a subset $K \subseteq \bL(F)$, we set 
\[
    \Vol_\sbZ(K) \coloneqq (\dim(F))! \cdot \vol(K)\,.
\]
Now let us fix two faces $F_\alpha,F_\beta$ of $P$ such that $F_\beta < F_\alpha$. By taking a suitable affine transformation of the lattice $M(\ka)$, we may assume that the origin $0$ of $M(\ka)$ is a vertex of the smaller face $F_\beta$. With this choice of the origin $0 \in F_\beta \cap M(\ka)$, we define the subsemigroup $\Salpha$ of $M_\alpha$ generated by $\ka \cap F_\alpha$. Although $\Salpha$ also depends on $F_\beta$, we denote it by $\Salpha$ for simplicity. Denote by $M_\alpha/F_\beta$ the quotient lattice $M_\alpha/(M_\alpha \cap \bL(F_\beta))$ of rank $\dim(F_\alpha) - \dim(F_\beta)$. 

We adapt the following definitions of the functions $u$ and $i$, and of the Euler obstruction, from \cite{GKZ94} and \cite{matsui2011geometric}.

\begin{definition}[$u$ and $i$ functions]\label{def: general-u}
We denote by $\Salpha/F_\beta$ the image of $\Salpha \subseteq M_\alpha$ in the quotient $\bZ$-lattice $M_\alpha/F_\beta$.
We denote by $K(\Salpha/F_\beta)$ (resp. $K_+(\Salpha/F_\beta)$) the convex hull of $\Salpha/F_\beta$ (resp. $(\Salpha/F_\beta) \setminus \{0\}$) in $(M_\alpha/F_\beta)_\sbR$ and set 
\[
    K_-(\Salpha/F_\beta) \coloneqq \overline{K(\Salpha/F_\beta) \setminus K_+(\Salpha/F_\beta)}\,.
\]
We call $K_-(\Salpha/F_\beta)$ the subdiagram part of the semigroup $\Salpha/F_\beta$ and denote by $u(F_\alpha,F_\beta)$ the normalized $(\dim(F_\alpha) - \dim(F_\beta))$-dimensional volume of $K_-(\Salpha/F_\beta)$ with respect to the $\bZ$-lattice $M_\alpha/F_\beta \subseteq (M_\alpha/F_\beta)_\sbR$.\footnote{The notation $u(F_\alpha,F_\beta)$ is equivalent to $u(\Salpha/F_\beta)$ in \cite{matsui2011geometric}.} If $F_\alpha = F_\beta$, we set $u(F_\alpha,F_\alpha) \coloneqq 1$. Finally, for two faces $F_\alpha,F_\beta$ of $P$ such that $F_\beta < F_\alpha$, we define $i(F_\alpha,F_\beta)$ to be the index $[M_\alpha \cap \bL(F_\beta): M_\beta]$.
\end{definition}

\begin{definition}\label{def: eu}
The \emph{Euler obstruction} of a face $F_\alpha$ of a polytope $P$ associated to a projective toric variety is defined recursively as follows:
\begin{enumerate}
    \item $\eu(P) \coloneqq 1$,
    \item $\eu(F_\beta) \coloneqq \sum_{F_\beta \lneq F_\alpha}(-1)^{\dim(F_\alpha) - \dim(F_\beta) - 1} i(F_\alpha, F_\beta) \cdot u(F_\alpha,F_\beta) \cdot \eu(F_\alpha)$.
\end{enumerate}
\end{definition}

\noindent See \cite{matsui2011geometric} for a discussion of how Definition~\ref{def: eu} relates to the Euler obstruction of a point on a variety.
With this collection of definitions in hand, we see a clear path to calculate $\deg(\nabla_{\!\ka})$: find expressions for the functions $u$ and $i$ (Section~\ref{subsec: functions u and i}), the Euler obstruction $\eu$ (Section~\ref{subsec: Euler obstruction}), the normalized volumes $\Vol_\sbZ$ (Section~\ref{subsec: normalized volumes}), and finally $\delta_i$ respectively for strictly balanced (Section~\ref{subsec: degree A-discriminant strictly balanced}) and on-boundary (Section~\ref{subsec: degree A-discriminant on-boundary}) game formats.

\subsection{The faces of balanced Nashtopes}

The first step toward applying Theorem \ref{thm:polytope} is to characterize the faces of the Nashtope. We do this by indexing the faces with auxiliary combinatorial objects, which we call \emph{face strings}.

\begin{definition}
Let $a \in \bZ^+$. An \emph{$a$-boxed set} is either
\begin{enumerate}
    \item $S$, a nonempty subset of $[a]$, or
    \item $\boxed{S}$, a nonempty subset of $[a]$ drawn inside a box, or
    \item an empty box $\boxed{\emptyset}$.
\end{enumerate}
\end{definition}

\begin{definition}
The \emph{cardinality} $|S|$ of an $a$-boxed set $S$ is 
\begin{enumerate}
    \item the number of elements of $S$ if $S$ is a boxed or unboxed nonempty subset of $[a]$,
    \item 1, if $S$ is an empty box.
\end{enumerate}
\end{definition}

\begin{definition}
We say that $x$ \emph{is an element of} an $a$-boxed set $S$ if $x$ is an element of the unboxed version of $S$.
\end{definition}
The term ``prismatic'' for the strings becomes clear later in Theorem~\ref{thm:general-bijection}; they correspond to faces that are products of simplices.
\begin{definition}\label{def:prismatic-face-string}
A \emph{prismatic face string} or \emph{prism string} is defined as follows.
Let $C_i$ be a nonempty subset of $[d_i - 1]$ and, for $j \in [n] \setminus \{i\}$, let $D_j$ be a nonempty subset of $[d_j]$.
The prismatic face string associated to these sets $D_1,\dots,D_{i-1},$ $C_i,$ $D_{i+1},\dots,D_n$ is
\[
    D_1 \cdots D_{i-1}\boxed{C_i} D_{i+1}\cdots D_n\,.
\]
A prismatic face string with a boxed set in $i$-th position is called \emph{$i$-prismatic face string}.
A \emph{vertex string} is a prismatic face string such that $C_i$ and all subsets $D_j$ for $j\neq i$ are singletons.
\end{definition}

\begin{definition}\label{def:non-prismatic-face-string}
A \emph{nonprismatic face string} of format $\bd = (d_1, \dots,d_n)$ is defined as follows. For $i \in [n]$, let $C_i$ be a subset of $[d_i - 1]$ such that at least two of $C_1,\dots,C_n$ are nonempty, and let $D_i$ be a nonempty subset of $[d_i]$. The nonprismatic face string associated to these sets $C_1,\dots,C_n,$ $D_1,\dots,D_n$ is
\[
\begin{array}{ccc}
    \boxed{C_1} & \cdots & \boxed{C_n}\\[2pt]
    D_1 & \cdots & D_n
\end{array}
\]
or, written more compactly, $(\boxed{C_1} \cdots \boxed{C_n}, D_1 \cdots D_n)$.
\end{definition}

\begin{definition}
A \emph{prism representative} of format $\bd = (d_1, \dots,d_n)$ is defined as follows. For a unique $j \in [n]$, let $C_j$ be a nonempty subset of $[d_j - 1]$.
For each $i \in [n]$, let $D_i$ be a nonempty subset of $[d_i]$ with $D_j = \{t_j\}$ for some $t_j \in [d_j]$. The prism representative associated to these sets $C_j,$ $D_1,\dots,D_n$ is
\[
\begin{array}{ccccccc}
    \boxed{\emptyset} & \cdots & \boxed{\emptyset} & \boxed{C_j} & \boxed{\emptyset} & \cdots & \boxed{\emptyset}\\[5pt]
    D_1 & \cdots & D_{j-1} & \{t_j\} & D_{j+1} & \cdots & D_n\,.
\end{array}
\]
\end{definition}

\begin{remark}\label{rmk: general-caveat}
Prism representatives are not in bijection with prism strings. In fact, there is a $d_j$-to-one map from prism representatives to prism strings given by replacing $t_j$ with $\boxed{C_j}$ and forgetting all the empty sets. The advantage of defining prism representatives is that it enables us to speak of ``general format'' for face strings and it makes certain aspects of the Euler obstruction sum and volume calculations easier to describe, such as Corollary \ref{cor:vol-format} and Lemma \ref{lem:vol-gf}.
\end{remark}

\begin{definition}\label{def: i-prism string}
Consider a nonprismatic face string $\sigma =(\boxed{C_1}\cdots\boxed{C_n},{D_1}\cdots{D_n})$ such that $C_i$ is nonempty. The \emph{$i$-prism string} of $\sigma$ is $\sigma^i\coloneqq D_1\cdots D_{i-1}\boxed{C_i}D_{i+1}\cdots D_n$, that is the $i$-prismatic string obtained from $\sigma$ by deleting all subsets $C_j$ for $j\neq i$ and replacing $D_i$ with $\boxed{C_i}$.
\end{definition}

\begin{definition}
A \emph{face string} of format $\bd = (d_1,\dots,d_n)$ is either a prismatic face string or a nonprismatic face string of format $\bd = (d_1,\dots,d_n)$.
\end{definition}

Before defining the map $\phi$ from face strings to faces, we first set conventions for the vector space in which the faces lie. For each $i\in[n]$, define $U^i=\bR^{d_i - 1}$ with basis $\{e_1^i,\dots,e_{d_i - 1}^i\}$ and similarly $V^i=\bR^{d_i}$ with basis $\{f_1^i,\dots,f_{d_i}^i\}$. Each face is contained in $\bigoplus_{i=1}^n(U^i\oplus V^i)$.

\begin{definition}\label{def:general-phi}
Given a face string for a balanced game format, we map it to a face of the corresponding Nashtope. Let this map be called $\phi$.
\begin{enumerate}
    \item If the face string is a vertex string $\sigma =s_1 \cdots s_{i-1} \boxed{t_i} s_{i+1} \cdots s_n$, then set
    \[
        \phi(\sigma) \coloneqq e_{t_i}^i + \sum_{j\neq i} f_{s_j}^{j}\,.
    \]
    \item If the face string is prismatic of the form $\sigma =D_1 \cdots D_{i-1}\boxed{C_i} D_{i+1}\cdots D_n$, then set
    \[
        \phi(\sigma) \coloneqq \conv(\phi(s_1 \cdots s_{i-1} \boxed{t_i} s_{i+1} \cdots s_n) \mid \text{$t_i\in C_i$, $s_j\in D_j$ $\forall\,j\neq i$}\})\,.
    \]
    \item If the face string is nonprismatic of the form $\sigma =(\boxed{C_1}\cdots\boxed{C_n},{D_1} \cdots{D_n})$, then set
    \[
        \phi(\sigma) \coloneqq \conv(\phi(\sigma^i)\mid\text{$i\in[n]$ and $C_i\neq\emptyset$})
    \]
    according to Definition~\ref{def: i-prism string}.
\end{enumerate}
\end{definition}

\begin{lemma}\label{lem:general-phi-well-defined}
Let $\sigma$ be a face string. Then its image under $\phi$ is a face of $P$.
\end{lemma}
\begin{proof}
Let the vertex set of $\phi(\sigma)$ be called $\ka_{\sigma}$. 
First, suppose that $\sigma = D_1 \cdots D_{i-1} \boxed{C_i} D_{i+1} \cdots D_n$ is prismatic. Then
\[
    \ka_{\sigma} = \left\{e_{t_i}^i + \sum_{j \in [n] \setminus \{i\}} f_{s_{j}}^{j}\ \bigg|\ \text{$t_i \in C_i$, $s_j \in D_j$ for $j \in [n] \setminus \{i\}$}\right\}\,.
\]
Suppose we have an element $v$ in $\phi(\sigma)$ which satisfies $v = \sum_{a \in \ka} q_a \cdot a$ for some coefficients $q_a \ge 0$ with $\sum_{a \in \ka} q_a = 1$ (where we identify the matrix $\ka$ with its set of columns). By the definition of a face, it suffices to show that for all choices of coefficients, $q_a = 0$ for all $a \in \ka \setminus \ka_{\sigma}$. Indeed, observe that for any $a \in \ka \setminus \ka_{\sigma}$, either the coefficient of $e^i_{t_i}$ in $a$ is $1$ for some $t_i \in [d_i - 1] \setminus C_i$, or the coefficient of $f_{s_{j}}^{j}$ in $a$ is $1$ for some $j \in [n] \setminus \{i\}$ and $s_{j} \in [d_i] \setminus D_i$. Since the coefficients of each such $e_{t_i}^i$ and $f_{s_j}^{j}$ in each element of $\ka_{\sigma}$ are all zero, we must have $q_a = 0$ for all such $a$.

Second, suppose that $\sigma$ is nonprismatic of the form $(\boxed{C_1} \cdots \boxed{C_n}, $ ${D_1} \cdots {D_n})$. Then 
\[
    \ka_{\sigma} = \bigcup_{i\colon C_i\neq\emptyset}\ka_{\sigma^i}
\]
where $\sigma^i$ is the $i$-prism string of $\sigma$.

\noindent
Suppose we have an element $v$ in $\phi(\sigma)$ which satisfies $v = \sum_{a \in \ka} q_a \cdot a$ for some coefficients $q_a \ge 0$ with $\sum_{a \in \ka} q_a = 1$. Similarly to before, it suffices to show that $q_a = 0$ for all $a \in \ka \setminus \ka_{\sigma}$. Indeed, observe that for any $a \in \ka \setminus \ka_{\sigma}$, either the coefficient of $e^i_{t_i}$ in $a$ is $1$ for some $i \in [n]$ and $t_i \in [d_i - 1] \setminus C_i$, or the coefficient of $f_{s_{j}}^{j}$ in $a$ is $1$ for some $j \in [n]$ and $s_{j} \in [d_i] \setminus D_i$. Since the coefficients of each such $e_{t_i}^i$ and $f_{s_{j}}^{j}$ in each element of $\ka_{\sigma}$ are all zero, we must indeed have $q_a = 0$ for all such $a$.
\end{proof}

\begin{example}
Consider the game format $\bd = (2,2,3)$ and the nonprismatic face string $\sigma =(\setlength{\arraycolsep}{2pt}\begin{array}{|c|c|c|}
\hline
\emptyset & 1 & [2] \\
\hline
\end{array},2\,2\,[3])$.
There is no $1$-prism string of $\sigma$. The $2$-prism string of $\sigma$ is $\sigma^2=2\,\boxed{1}\,[3]$. In particular, $\sigma^2$ has vertex strings $2\,\boxed{1}\,1$, $2\,\boxed{1}\,2$, and $2\,\boxed{1}\,3$, and for example $\phi(2\,\boxed{1}\,3)= e_1^2 + f_2^1 + f_3^3$. The $3$-prism string of $\sigma$ is $\sigma^3 = 2\,2\,\boxed{[2]}$. In particular, $\sigma^3$ has vertex strings $2\,2\,\boxed{1}$ and $2\,2\,\boxed{2}$. In particular, this nonprismatic face is generated by five vertex strings, three of them generate the two-dimensional face $\sigma^2$, and the remaining two generate the edge $\sigma^3$.
\end{example}

\begin{definition}
A \emph{prismatic} face is one whose face string is prismatic.
\end{definition}

\begin{lemma}\label{lem:general-cube-intersect}
Let $F$ be any face of $P$, and let $i \in [n]$. Let $\ka_i$ be the vertex set of the face string $[d_1]\cdots [d_{i-1}] \boxed{[d_i - 1]}[d_{i+1}] \cdots [d_n]$. Suppose $\ka_F \cap \ka_i \neq \emptyset$. Then $\ka_F \cap \ka_i$ is the vertex set of some $i$-prismatic face string.  
\end{lemma}
\begin{proof}
Let $F_i$ be the convex hull of $\ka_i$. By Lemma \ref{lem:general-phi-well-defined}, $F_i$ is a face of $P$. Since faces of $P$ form a lattice, it follows that $F \cap F_i$ is a face of $P$. Since $F \cap F_i \subseteq F_i$, we have that $F \cap F_i$ is also a face of $F_i$. But $F_i$ is just a product of simplices, and we can confirm that its faces are precisely the convex hulls of prismatic face strings with some $\boxed{C_i}$'s.
Thus, $\ka_F \cap \ka_i$ is the vertex set of some prismatic face string with a $\boxed{C_i}$.
\end{proof}

\begin{lemma}\label{lem:general-cube-overlap}
Let $\sigma = D_1\cdots D_{i-1}\,\boxed{C_i}\,D_{i+1}\cdots D_n$ be an $i$-prismatic face string, and let $\sigma' = D'_1\cdots D'_{j-1}\boxed{C'_{j}}D'_{j+1}\cdots D'_n$ be a $j$-prismatic face string with $i < j$, and let $F''$ be the convex hull of $\ka_{\sigma} \cup \ka_{\sigma'}$. Suppose that $F''$ is a face of $P$. Then $D_\ell = D'_\ell$ for all $\ell \in [n] \setminus \{i, j\}$.
\end{lemma}
\begin{proof}
Let $v$ and $v'$ be the centroids of $\sigma$ and $\sigma'$, respectively, and let $w \coloneqq \frac{v+v'}{2}$. Then, clearly, $w \in F''$. Define
\begin{align*}
    \tilde{\sigma} &\coloneqq D'_1 \cdots D'_{i-1} \boxed{C_i} D'_{i+1} \cdots D'_{j-1} D_j D'_{j+1} \cdots D'_n\\
    \tilde{\sigma}' &\coloneqq D_1 \cdots D_{i-1} D'_i D_{i+1} \cdots D_{j-1} \boxed{C'_j} D_{j+1} \cdots D_n
\end{align*}
and let $\tilde{v}$ and $\tilde{v}'$ be the centroids of $\tilde{\sigma}$ and $\tilde{\sigma}'$, respectively. Then, by comparing coefficients of $e^\ell_{t_\ell}$ and $f^\ell_{s_\ell}$ for each $\ell$, we find that $\frac{\tilde{v} + \tilde{v}'}{2} = w$ as well. Now, we have $w = \frac{1}{2| \ka_{\tilde{\sigma}}|}\sum_{a \in \ka_{\tilde{\sigma}}} a + \frac{1}{2| \ka_{\tilde{\sigma}'}|}\sum_{a \in \ka_{\tilde{\sigma}'}} a$. Since $F''$ is a face of $P$, it follows that $\ka_{\tilde{\sigma}} \cup \ka_{\tilde{\sigma}'} \subseteq \ka_{{\sigma}} \cup \ka_{{\sigma}'}$. By Definition~\ref{def:general-phi} of $\phi$, we have that $\ka_{\tilde{\sigma}} \cap \ka_{\sigma'} = \ka_{\tilde{\sigma}'} \cap \ka_{\sigma} = \emptyset$, so we deduce that $\ka_{\tilde{\sigma}} \subseteq \ka_{\sigma}$ and $\ka_{\tilde{\sigma}'} \subseteq \ka_{\sigma'}$. The first inclusion implies that $D'_\ell \subseteq D_\ell$ for all $\ell \in [n] \setminus \{i, j\}$, and the second inclusion implies $D_\ell \subseteq D'_\ell$ for all $\ell \in [n] \setminus \{i, j\}$.
\end{proof}

\begin{theorem}\label{thm:general-bijection}
The map $\phi$ in Definition \ref{def:general-phi} is a bijection between faces and face strings. 
\end{theorem}
\begin{proof}
By Lemma \ref{lem:general-phi-well-defined}, the image of $\phi$ is always a face of $P$, and it is clear from the definition of $\phi$ that it is injective. Conversely, suppose $F$ is a face of $P$, and for each $i \in [n]$, let $\ka_i$ be as in Lemma \ref{lem:general-cube-intersect}. In addition, let $\emptyset \neq B \subseteq [n]$ be the set of indices $i$ such that $\ka_F \cap \ka_i \neq \emptyset$. Then by Lemma \ref{lem:general-cube-intersect}, there exist prism strings $\sigma^i$ for $i \in B$, with $\sigma^i$ having a $\boxed{C_i}$, such that $\ka_F \cap \ka_i = \ka_{\sigma^i}$ for all $i$. If $B = \{i\}$ is a singleton, this immediately tells us that $F = \phi(\sigma^i)$. Otherwise, note that given two distinct $i,j\in B$, the convex hull of $\ka_i \cup \ka_j$, which corresponds to the face string
\[
    (\emptyset \cdots \emptyset\,\boxed{[d_{\min\{i,j\}} - 1]}\,\emptyset \,\cdots\,\emptyset\,\boxed{[d_{\max\{i,j\}} - 1]}\,\emptyset\,\cdots\,\emptyset,[d_1]\,\cdots\,[d_n])
\] 
is a face of $P$ by Lemma \ref{lem:general-phi-well-defined}. By intersecting $F$ with this face, we deduce that the convex hull of $\ka_{\sigma^i} \cup \ka_{\sigma^j}$ is a face of $P$ as well, from which we deduce by Lemma \ref{lem:general-cube-overlap} that the letters of $\sigma^i$ and $\sigma^j$ agree whenever neither letter is a set with a box around it. Consequently, there exists a single list of sets $D_1,\dots,D_n$ such that, for each $i\in B$, we have $\sigma^i = D_1 \cdots D_{i-1} \boxed{C_i} D_{i+1} \cdots D_n$. If we let $\hat{\sigma}$ be the string $(\boxed{C_1}\cdots\boxed{C_n}, D_1\cdots D_n)$, where $C_\ell \coloneqq \emptyset$ for all $\ell \notin B$, we conclude that $\ka_{\hat{\sigma}} = \bigcup_{i \in B} \ka_{\sigma^i} = \ka_F$, hence that $F = \phi(\hat{\sigma})$.
\end{proof}

\begin{definition}
Let $\sigma =D_1\cdots D_{i-1}\,\boxed{C_i}\,D_{i+1}\cdots D_n$ be a prismatic face string. The \emph{dimension} of $\sigma$ is
\[
    \dim(\sigma) \coloneqq |C_i| - 1 + \sum_{j\neq i}(|D_j| - 1)\,.
\]
Let $\sigma =(\boxed{C_1}\cdots\boxed{C_n},{D_1}\cdots{D_n})$ be a nonprismatic face string. The \emph{dimension} of $\sigma$ is
\[
    \dim(\sigma) \coloneqq -1 + \sum_{i=1}^n (|D_i| - 1 + |C_i|)\,.
\]
\end{definition}

\noindent Given face strings $\sigma$ and $\sigma'$, write $\sigma \le \sigma'$ if $\phi(\sigma) \subseteq \phi(\sigma')$. 

\begin{lemma}\label{lem: dim face strings}
The dimension of a face string equals the dimension of its image under $\phi$.
\end{lemma}
\begin{proof}
By Theorem~\ref{thm:general-bijection}, the face lattice is isomorphic to the lattice formed by face strings (together with a bottom element corresponding to the empty face) with the relation $\le$. Given two prismatic face strings $\sigma = D_1 \cdots D_{i-1} \boxed{C_i} D_{i+1} \cdots D_n$ and $\sigma' = D'_1 \cdots D_{j-1}' \boxed{C_j} D_{j+1}' \cdots D'_n$, one observes that $\sigma \le \sigma'$ if and only if $i = j$, $C_i \subseteq C_j$, and $D_\ell \subseteq D'_\ell$ for all $\ell \in [n] \setminus \{i\}$. Meanwhile, given a prismatic face string $\sigma = D_1 \cdots D_{i-1} \boxed{C_i} D_{i+1} \cdots D_n$ and a nonprismatic face string $\sigma' = (\boxed{C'_1}\cdots\boxed{C'_n},D'_1\cdots D'_n)$, it is clear that $\sigma \le \sigma'$ if and only if $C_i \subseteq C'_i$ and $D_\ell \subseteq D'_\ell$ for all $\ell \in [n] \setminus \{i\}$. Finally, given two nonprismatic face strings $\sigma = (\boxed{C_1}\cdots\boxed{C_n},D_1\cdots D_n)$ and $\sigma' = (\boxed{C'_1}\cdots\boxed{C'_n},D'_1\cdots D'_n)$, we see that $\sigma \le \sigma'$ if and only if $C_\ell \subseteq C'_\ell$ and $D_\ell \subseteq D'_\ell$ for all $j \in [n]$. It is impossible to have $\sigma \le \sigma'$ if $\sigma$ is nonprismatic and $\sigma'$ is prismatic.

Let $\sigma$ be a face string, and proceed by induction on the dimension of $\sigma$. If the dimension of $\sigma$ is $0$, then $\sigma$ is necessarily a prismatic face string of the form $D_1 \cdots D_{i-1} \boxed{C_i} D_{i+1} \cdots D_n$ for some singleton sets $D_\ell$ and $C_i$, and we verify based on the characterization in the above paragraph that no face string can be strictly less than $\sigma$. Next, let $d \ge 0$, and suppose that for all face strings $\sigma'$ of dimension $d$, the face $\phi(\sigma')$ has dimension $d$. Let $\sigma$ be a face string of dimension $d+1$. Suppose first that $\sigma$ is prismatic of the form $D_1 \cdots D_{i-1} \boxed{C_i} D_{i+1} \cdots D_n$. Then since $|C_i|-1 + \sum_{\ell \in [n]\setminus\{i\}}(|D_\ell|-1) = d+1 \ge 1$, at least one of $C_i$ or the $D_\ell$ must have more than one element. Letting $\sigma'$ denote the face string obtained by removing an element from such a set, we see that $\sigma'$ has dimension $d$, so $\phi(\sigma')$ has dimension $d$ by the inductive hypothesis. Since $\sigma$ covers $\sigma'$, it follows that $\dim(\phi(\sigma)) = d+1$.

Now suppose that $\sigma$ is nonprismatic of the form $(\boxed{C_1}\cdots\boxed{C_n},D_1,\dots,D_n)$. Suppose first that (i) it is possible to remove an element from one of the $C_i$ while leaving at least two of the $C_i$ nonempty, or (ii) it is possible to remove an element from one of the $D_i$ while leaving all of the $D_i$ nonempty. Let $\sigma'$ be the nonprismatic face string that results from applying the removal in (i) or (ii). Then it is easy to see that $\sigma'$ has dimension $d$ and that $\sigma$ covers $\sigma'$. Once again, we may apply the inductive hypothesis to deduce $\dim(\phi(\sigma)) = d+1$.  

Finally, suppose that the $D_\ell$ are all singletons, and that there exist some distinct $i, j$ such that $C_\ell = \emptyset$ for all $\ell \notin \{i, j\}$ and $C_i$ and $C_j$ are singletons. Let $\sigma'$ denote the prismatic face string $D_1 \cdots D_{i-1} \boxed{C_i} D_{i+1} \cdots D_n$. Then $\sigma'$ has dimension $d = 0$, so $\phi(\sigma')$ has dimension $0$. Once again, it is easy to verify that $\sigma$ covers $\sigma'$, so we conclude that $\dim(\phi(\sigma)) = 1 = d+1$.
\end{proof}

\noindent Note that whereas a prismatic face can be contained in a prismatic or a nonprismatic face, a nonprismatic face can only be contained in other nonprismatic faces.

\begin{example}
Let $\bd = (2,2,3)$.
Consider the nonprismatic face strings $\sigma = (\setlength{\arraycolsep}{2pt}\begin{array}{|c|c|c|}\hline 1 & 1 & [2] \\\hline\end{array}, 1 [2] \{2,3\})$ and $\tilde{\sigma} = (\setlength{\arraycolsep}{2pt}\begin{array}{|c|c|c|}\hline 1 & 1 & [2] \\\hline\end{array}, 1 [2] 3)$ and their $1$-prism strings $\sigma^1 = \boxed{1}[2]\{2,3\}$ and $\tilde\sigma^1 = \boxed{1}[2]3$ respectively. Then $\tilde\sigma$ is covered by $\sigma$ and $\tilde\sigma^1$ is covered by $\sigma^1$ in the poset induced by the face lattice. Clearly also $\sigma^1 < \sigma$ and $\tilde\sigma^1 < \tilde\sigma$. We have $\dim(\sigma)=5$, $\dim(\tilde \sigma)=4$, $\dim(\sigma^1)=2$ and $\dim(\tilde \sigma^1) = 1$.
\end{example}

\subsection{The functions \texorpdfstring{$u$}{u} and \texorpdfstring{$i$}{i} for Nashfaces}\label{subsec: functions u and i}

{As seen in Definition \ref{def: eu}, the Euler obstruction of a Nashface is constructed recursively using the functions $u$ and $i$ as introduced in Definition~\ref{def: general-u}. Both depend on two faces of the Nashtope, one contained in the other. To compute the Euler obstruction in the following section, we first compute $u$ and $i$ for different pairs of Nashfaces. We start out by constructing affine bases for the affine semigroups generated by both prismatic and nonprismatic Nashfaces.}

\begin{remark}
Readers familiar with \cite{matsui2011geometric} may recognize $i(F_\alpha,F_\beta)\cdot u(F_\alpha,F_\beta)$ as the relative subdiagram volume $\RSV_\sbZ(F_\alpha,F_\beta)$, given in Definition 4.2. However, the literature does not agree on a universal definition of $ \RSV_\sbZ$. For example, $\RSV_\sbZ$ in \cite{nodland2018eulerobstruction} is equivalent to $u$ in \cite{matsui2011geometric}. For this reason, we refrain from using the notation $\RSV_\sbZ$ in this paper. Nonetheless, much of the intuition about $u$-invariant operations in this section comes from thinking of $u$ as a relative subdiagram volume.
\end{remark}

\noindent Given a $\bZ$-lattice (free $\bZ$-module) $M$, the zero element of the underlying abelian group is called the \emph{origin}.

\begin{lemma}\label{lem:general-prismatic-lattice}
Let $F = D_1 \cdots D_{i-1} \boxed{C_i} D_{i+1} \cdots D_n$ be an $i$-prismatic face string, and fix a vertex string $v = \delta_1 \cdots \delta_{i-1} \boxed{\gamma_i} \delta_{i+1} \cdots \delta_n$ such that $\phi(v) $ is a vertex of $ \phi(F)$. Then the affine semigroup generated by $\phi(F)$ with origin $\phi(v)$ is affinely spanned by vertices obtained from applying $\phi$ to the following vertex strings:
\begin{itemize}
    \item for each $j \in [n] \setminus \{i\}$ and $\delta_{j}' \in D_j \setminus \{\delta_{j}\}$, the vertex string $v_{j, \delta'_{j}}$ obtained from $v$ by replacing $\delta_{j}$ with $\delta_{j}'$;
    \item for each $\gamma'_i \in C_i \setminus \{\gamma_i\}$, the vertex string $v_{i, \gamma'_i}$ obtained from $v$ by replacing $\gamma_i$ with $\gamma'_i$.
\end{itemize}
Moreover, these vertices form an affine basis for the affine semigroup generated by $\phi(F)$ with origin $\phi(v)$.
\end{lemma}
\begin{proof}
First, we check that the claimed spanning set lies in the affine semigroup generated by $\phi(F)$. Since $\phi(v)$ is a vertex of $\phi(F)$, we have $\delta_\ell \in D_\ell$ for each $\ell \neq i$ and $\gamma_i \in C_i$. By construction, $v_{j, \delta'_j}$ also satisfies $\delta_\ell \in D_\ell$ for each $\ell \neq i$ and $\gamma_i \in C_i$, so $\phi(v_{j, \delta'_j})$ is a vertex of $\phi(F)$. A similar argument shows $\phi(v_{i, \gamma'_i})$ is also a vertex of $\phi(F)$. Thus, the claimed spanning set lies in the affine semigroup. 

Now, we verify that the vectors in the claimed affine spanning set span the affine semigroup. It suffices to show that every vertex of $\phi(F)$ -- equivalently, a generator of the affine semigroup -- can be expressed as a nonnegative linear combination of $\phi(v_{i, \gamma'_i})$ and the $\phi(v_{j, \delta'_j})$. Choose a vertex of $\phi(F)$ and let $w = \delta'_1 \cdots \delta'_{i-1} \boxed{\gamma'_i} \delta'_{i+1} \cdots \delta'_n$ be its vertex string.
Then 
\[
    \phi(w) - \phi(v) = \big(\phi(v_{i,\gamma'_i}) - \phi(v)\big) + \sum_{j \neq i\,\mid\,\delta'_j \ne \delta_j} \left(\phi(v_{i_\ell, \delta'_{j_\ell}}) - \phi(v)\right)\,,
\]
where the first term is omitted if $\gamma'_i = \gamma_i$.
These vertices form an affine basis for the affine semigroup generated by $\phi(F)$ with origin $\phi(v)$ because the number of vertices in the bullet points equals the dimension of $\phi(F)$.
\end{proof}

\begin{lemma}\label{lem:general-nonprismatic-lattice}
Let $F = (\boxed{C_1}\cdots \boxed{C_n}, D_1 \cdots D_n)$ be a nonprismatic face string, and fix a vertex string $v = \delta_1 \cdots \delta_{i-1} \boxed{\gamma_i} \delta_{i+1} \cdots \delta_n$ such that $\phi(v) $ is a vertex of $ \phi(F)$. Then the affine semigroup generated by $\phi(F)$ with origin $\phi(v)$ is affinely spanned by vertices obtained from applying $\phi$ to the following vertex strings:
\begin{itemize}
    \item for each $j \in [n] \setminus \{i\}$ and $\delta'_j \in D_j \setminus \{\delta_j\}$, the vertex string $v_{j, \delta'_j}$ obtained from $v$ by replacing $\delta_j$ with $\delta'_j$;
    \item for each $\gamma'_i \in C_i \setminus \{\gamma_i\}$, the vertex string $v_{i, \gamma'_i}$ obtained from $v$ by replacing $\gamma_i$ with $\gamma'_i$;
    \item for each $j\in [n] \setminus \{i\}$, $\gamma_j \in C_j$, and $\delta_i \in D_i$, the vertex string $v_{j, \gamma_j\,;\,i, \delta_i}$ obtained from $v$ by replacing $\delta_j$ with $\boxed{\gamma_j}$ and $\boxed{\gamma_i}$ with $\delta_i$.
\end{itemize}
Moreover, if we fix a particular element $\delta_i \in D_i$, then the affine vector space generated by $\phi(F)$ with origin $\phi(v)$ has the following vectors as an affine basis:
\begin{itemize}
    \item for each $j \in [n]$ and $\delta'_j \in D_j \setminus \{\delta_j\}$, the vector $\phi(v) + f_{\delta'_j}^j - f_{\delta_j}^j$;
    \item for each $\gamma'_i \in C_i \setminus \{\gamma_i\}$, the vector $\phi(v) + e_{\gamma'_i}^i - e_{\gamma_i}^i$;
    \item for each $j \in [n] \setminus \{i\}$ and $\gamma'_j \in C_j$, the vector $\phi(v) + e_{\gamma'_j}^j - e_{\gamma_i}^i$.
\end{itemize}
\end{lemma}
\begin{proof}
The proof of the first claim is similar to that of Lemma \ref{lem:general-prismatic-lattice}: by construction, the three types of bullet-pointed vertices are in fact vertices of $\phi(F)$; also, any vertex string $w$ corresponding to a vertex of $\phi(F)$ can be reached from $v$ by performing a finite number of ``moves,'' namely changing singleton $\delta_i$'s and possibly changing the index of $\gamma_j$, each of which contributes one summand to $\phi(w) - \phi(v)$. The spanning portion of the second claim follows by noting that any vertex of $\phi(F)$ can be expressed as $v$ plus some linear combination of the vectors $f_{\delta'_j}^j - f_{\delta_j}^j$, $e_{\gamma'_i}^i - e_{\gamma_i}^i$, and $e_{\gamma'_j}^j - e_{\gamma_i}^i$ appearing in the second set of bullet points. Linear independence then follows because the number of vectors in the second set of bullet points equals the dimension of $\phi(F)$.
\end{proof}

\begin{proposition}\label{prop:Kplus-2}
Let $V$ be a finite-dimensional vector space and $A \subseteq V \setminus \{0\}$ a finite set. Let $\cone(A) \coloneqq \{ \sum_{a \in A} c_a \cdot a \mid \text{$c_a \ge 0$ $\forall a\in A$}\}$. Define $S \coloneqq \bN A$ to be the finitely generated semigroup, $K \coloneqq \conv(S)$, and $K^+ \coloneqq \conv(S \setminus \{0\})$. Then
\begin{enumerate}
    \item $K = \cone(A)$,
    \item $K^+ = \left\{ \sum_{a \in A} c_a \, a \mid\text{$c_a \ge 0$ and $\sum_{a \in A} c_a \ge 1$}\right\}$.
\end{enumerate}
\end{proposition}
\begin{proof}
$(1)$ We first show that $K \subseteq \cone(A)$. Let $x \in K$. By the definition of the convex hull, there exist $x_1, \dots, x_n \in S$ and constants $b_1, \dots, b_n \ge 0$ summing to $1$ such that $x = b_1x_1 + \dots + b_nx_n$. For each $i$, write $x_i = \sum_{a \in A} m_a^{(i)} \cdot a$ for some nonnegative integers $m_a^{(i)}$. Then $x = \sum_{a \in A} (\sum_{i=1}^n b_i m_a^{(i)}) \cdot a$. Setting $c_a \coloneqq \sum_{i=1}^n b_i m_a^{(i)}$ for all $a$, we verify that $c_a \ge 0$ for all $a$.
    
We now show that $\cone(A)\subseteq K$. Let $c_a \ge 0$ for all $a \in A$, and set $x \coloneqq \sum_{a \in A} c_a \cdot a$. Define $c \coloneqq \sum_{a \in A} c_a$, and choose an integer $m \ge 0$ such that $c \le m$. Then we may write $x = \frac{m-c}{m} \cdot 0 + \sum_{a \in A} \frac{c_a}{m} \cdot ma$, telling us that $x$ is in the convex hull of $\{0\} \cup \{ma \mid a \in A\}$, hence in $K$.

$(2)$ We first show the inclusion $K^+ \subseteq \left\{ \sum_{a \in A} c_a \, a \mid\text{$c_a \ge 0$ and $\sum_{a \in A} c_a \ge 1$}\right\}$. Let $x \in K^+$. Then, by the definition of the convex hull, there exist $x_1, \dots, x_n \in S \setminus \{0\}$ and constants $b_1, \dots, b_n \ge 0$ summing to $1$ such that $x = b_1x_1 + \dots + b_nx_n$. For each $i$, write $x_i = \sum_{a \in A} m_a^{(i)} \cdot a$ for some nonnegative integers $m_a^{(i)}$, not all zero. Then $x = \sum_{a \in A} (\sum_{i=1}^n b_i m_a^{(i)}) \cdot a$. Setting $c_a \coloneqq \sum_{i=1}^n b_i m_a^{(i)}$ for all $a$, we verify that $c_a \ge 0$ for all $a$ as before. Moreover, observe that since for each $i$, at least one of the $m_a^{(i)}$ is nonzero, we may write $\sum_{a \in A} c_a = \sum_{i=1}^n b_i \cdot \sum_{a \in A} m_a^{(i)} \ge \sum_{i=1}^n b_i \cdot 1 = 1$.
    
We show the reverse inclusion $\left\{ \sum_{a \in A} c_a \, a \mid\text{$c_a \ge 0$ and $\sum_{a \in A} c_a \ge 1$}\right\}\subseteq K^+$. Let $c_a \ge 0$ for all $a \in A$ with $c \coloneqq \sum_{a \in A} c_a \ge 1$, and let $x \coloneqq \sum_{a \in A} c_a \cdot a$. Choose an integer $m > c$, and observe that $\frac{m-c}{m-1} + \frac{c-1}{m-1} \cdot m = \frac{m-c+mc-m}{m-1} = \frac{c(m-1)}{m-1} = c$. Thus, we may write
\[
    x = \sum_{a \in A} \left(\frac{m-c}{m-1} \cdot \frac{c_a}{c}\right) \cdot a + \sum_{a \in A} \left(\frac{c-1}{m-1} \cdot \frac{c_a}{c}\right) \cdot ma\,.
\]
Observe that $ma \in S \setminus \{0\}$, so the parenthesized expressions above can be viewed as coefficients of elements of $S \setminus \{0\}$. These coefficients are all nonnegative. Moreover, their sum is given by
\[
    \sum_{a \in A} \left(\frac{m-c}{m-1} + \frac{c-1}{m-1}\right) \cdot \frac{c_a}{c} = \sum_{a \in A} \frac{c_a}{c} = 1.
\]
Thus, $x$ is a convex combination of elements of $S \setminus \{0\}$, and so $x$ is in $K^+$.
\end{proof}

Recall that a prismatic face can be contained in a prismatic or a nonprismatic face, but a nonprismatic face can only be contained in other nonprismatic faces.
Using Proposition \ref{prop:Kplus-2}, we now compute $u(F,F')$ for all pairs $F' \leq F$ of prismatic or nonprismatic faces of the Nashtope.  

\begin{lemma}\label{lem:nonprismatic-volume-explicit}
Consider two nonprismatic face strings $F' \le F$, and choose a vertex string
\[
    v = \delta_1 \cdots \delta_{i-1} \boxed{\gamma_i} \delta_{i+1} \cdots \delta_n
\]
of $F'$. Let $W \coloneqq V/\bL(F')$, with the image of $v$ treated as zero, where $V$ is the span of the collection of vertices corresponding to the first three bullet-pointed vertex strings in Lemma \ref{lem:general-nonprismatic-lattice}, with projection map $p\colon V \to W$, and let $w' = p(\phi(v))$ and $G = p(\ka_F)$. Let $K \subseteq W$ be the convex hull of the semigroup $S_G$ generated by $G$, where $w'$ is the additive identity, and let $K^+$ be the convex hull of $S_G \setminus \{w'\}$. Then
\[
    \Vol_{M(G)} (K \setminus K^+) = 1\,.
\]
\end{lemma}
\begin{proof}
We show that $S_G$ has a linearly independent generating set and conclude that the normalized lattice volume is $1$.
Let $F = (\boxed{C_1} \cdots \boxed{C_n}, $ ${D_1} \cdots {D_n})$ and let $F' = (\boxed{C'_1} \cdots \boxed{C'_n}, $ ${D'_1} \cdots {D'_n})$, and choose an arbitrary element $\delta_i \in D'_i$, an arbitrary index $i_0 \neq i$ such that $C_{i_0}' \neq \emptyset$, and an arbitrary element $\gamma_{i_0} \in C_{i_0}'$. By Lemma \ref{lem:general-nonprismatic-lattice} applied to $F'$, in the space $\im(p) = W$,
\[
    \begin{cases}
        p(\phi(v)+f_{\delta'_j}^j-f_{\delta_j}^j) = 0 & \text{for $j \in [n]$ and $\delta'_j \in D'_j$} \\
        p(\phi(v) + e_{\gamma'_i}^i - e_{\gamma_i}^i) = 0 & \text{for $\gamma'_i \in C_i$} \\
        p(\phi(v) + e_{\gamma'_j}^j - e_{\gamma_i}^i) = 0 & \text{for $j \in [n] \setminus \{i\}$ and $\gamma'_j \in C'_j$.}
    \end{cases}
\]
We claim that $S_G$ is generated by
\[
    V' = 
    \begin{cases}
        p(\phi(v_{j, \delta'_j})) & \text{for $j \in [n] \setminus \{i\}$ and $\delta'_j \in D_j \setminus D'_j$}, \\
        p(\phi(v_{i_0,\gamma_{i_0}\,;\,i,\delta'_i})) & \text{for $\delta'_i \in D_i \setminus D'_i$} \\
        p(\phi(v_{i, \gamma'_i}))  &\text{ for $\gamma'_i \in C_i \setminus C'_i$} \\
        p(\phi(v_{j, \gamma_j\,;\,i,\delta_i}))  & \text{for $j \in [n] \setminus \{i\}$ and $\gamma_j \in C_j \setminus C'_j$.}
    \end{cases}
\]
The argument of Lemma \ref{lem:general-nonprismatic-lattice} shows that all vectors in $V'$ belong to $S_G - w'$. To see that $V'$ generates $S_G$, it suffices to show that $p \circ \phi$ applied to every vector in the first set of bullet points of Lemma~\ref{lem:general-nonprismatic-lattice} is a nonnegative linear combination of elements of $V'$. The first and second bullet points are essentially automatic. For the third bullet point, let $j \in [n] \setminus \{i\}$, $\gamma_j \in C_j$, and $\delta'_i \in D_i$. If $\gamma_j \in C'_j$ and $\delta'_i \in D'_i$, then $p(\phi(v_{j,\gamma_j\,;\,i,\delta'_i})) = 0$. If $\gamma_j \in C_j \setminus C'_j$ and $\delta'_i \in D'_i$, then we can verify that
\[
    p(\phi(v_{j,\gamma_j\,;\,i,\delta'_i})) = p(\phi(v_{j,\gamma_j\,;\,i,\delta_i}) +f_{\delta'_i}^i-f_{\delta_i}^i) = p(\phi(v_{j,\gamma_j\,;\,i,\delta_i}))\,.
\]
If $\gamma_j \in C'_j$ and $\delta'_i \in D_i \setminus D'_i$, then
\[
    p(\phi(v_{j,\gamma_j\,;\,i,\delta'_i})) = p(\phi(v_{i_0,\gamma_{i_0}\,;\,i,\delta'_i}) + (e_{\gamma_j}^j-e_{\gamma_i}^i)-(e_{\gamma_{i_0}}^{i_0}-e_{\gamma_i}^i)) = p(\phi(v_{i_0,\gamma_{i_0}\,;\,i,\delta'_i}))\,.
\]
Finally, if $\gamma_j \in C_j \setminus C'_j$ and $\delta'_i \in D_i \setminus D'_i$, then we can similarly verify that
\[
    p(\phi(v_{j,\gamma_j\,;\,i,\delta'_i})) = p(\phi(v_{j,\gamma_j\,;\,i,\delta_i})) + p(\phi(v_{i_0,\gamma_{i_0}\,;\,i,\delta'_i}))\,.
\]
We must now show that $\Vol_{M(G)}(K \setminus K^+) = 1$. To do this, observe that the elements of $V'$ are linearly independent inside $W$. Thus, if we let $\tilde{e}_1, \dots, \tilde{e}_m$ be the elements of $V'$ listed in some order, then $S_G-w'$ is the semigroup inside $\bR^m$ generated by $\tilde{e}_1,\dots,\tilde{e}_m$, and $M(G)-w'$ is the lattice $\bZ\tilde{e}_1 + \dots + \bZ\tilde{e}_m$. In this new basis, by Proposition \ref{prop:Kplus-2}, $K$ is equal to the set of vectors $a_1\tilde{e}_1+\dots+a_m\tilde{e}_m$ with $a_1,\dots,a_m \ge 0$, whereas $K^+$ is equal to the set of these vectors which additionally satisfy $a_1+\dots+a_m \ge 1$. Thus, $K \setminus K^+$ is the set of vectors $a_1\tilde{e}_1+\dots+a_m\tilde{e}_m$ with $a_1,\dots,a_m \ge 0$ and $a_1+\dots+a_m < 1$. The volume of $K \setminus K^+$ according to the Lebesgue measure in the coordinates given by $\tilde{e}_1,\dots,\tilde{e}_m$ is $\frac{1}{m!}$. Since $M(G)-w'$ is just $\bZ^m$ in our new coordinate system, we can multiply the aforementioned volume of $K \setminus K^+$ by $m!$ to get that the volume with respect to the lattice $\bZ^m$ is $\frac{1}{m!} \cdot m! = 1$.
\end{proof}

\begin{lemma}\label{lem:non-prismatic-u}
For two nonprismatic Nashfaces $F,F'$ with $F' \le F$, we have $u(F,F') = 1$.
\end{lemma}
\begin{proof}
This is a restatement of Lemma \ref{lem:nonprismatic-volume-explicit}. In particular, using the definition of $u$ given in Definition \ref{def: general-u}, we are taking $F = F_\alpha$, $F' = F_\beta$, $M(G) = M_\alpha/F_\beta$, $S_G = \Salpha/F_\beta $.
\end{proof}

\begin{lemma}\label{lem:prismatic-u}
For two prismatic Nashfaces $F,F'$ with $F' \le F$, we have $u(F,F') = 1$.
\end{lemma}
\begin{proof}
By Lemma \ref{lem:general-prismatic-lattice}, the affine semigroup generated by $\phi(F)$ with origin $\phi(v)$ is affinely generated by the vertex strings in the two bullet points corresponding to $F$, and the affine subspace spanned by $\phi(F')$ with origin $\phi(v)$ is affinely spanned by the vertex strings in the two bullet points corresponding to $F'$. Consequently, the image of the semigroup generated by $F$ in the quotient is just the free semigroup generated by the images of all the vertex strings in the two bullet points corresponding to $F$ that do not appear in the bullet points corresponding to $F'$. More precisely, the vertex strings in the previous sentence are: for each $j \in [n] \setminus \{i\}$ and $\delta'_j \in D_j \setminus D'_j$, the vertex string $v_{j,\delta'_j}$; and for each $\gamma'_i \in C_i \setminus C'_i$, the vertex string $v_{i,\gamma'_i}$. Now apply \cite[Chapter 5, Example 3.9]{GKZ94}. 
\end{proof}

\begin{proposition}\label{prop:vol-mult-2}
Let $V_1, V_2$ be finite-dimensional vector spaces, and for 
$i \in \{1,2\}$ let $A_i \subseteq V_i \setminus \{0\}$ be a finite set of nonzero vectors. Set $V = V_1 \oplus V_2$, $A = (A_1 \times \{0\}) \cup (\{0\} \times A_2)$, and let $K = \cone(A)$ and $K_i = \cone(A_i)$. Define $K_i^+$ and $K^+$ as in Proposition~\ref{prop:Kplus-2}. Then
\[
    \Vol_\sbZ(K \setminus K^+) = \Vol_\sbZ(K_1 \setminus K_1^+) \cdot \Vol_\sbZ(K_2 \setminus K_2^+)\,.
\]
\end{proposition}
\begin{proof}
Let $K_i^- \coloneqq K_i \setminus K_i^+$ and $K^- \coloneqq K \setminus K^+$, and define a map
\[
    t_1 \colon K_1^- \to [0,1)\,,\quad t_1(v) \coloneqq \inf\{c \in [0,1) \mid v \in cK_1^-\}\,.
\]
Then define a map
\[
    K_1^- \times K_2^- \to K^-\,,\quad (v_1, v_2) \mapsto v_1 + (1-t_1(v_1))v_2\,.
\]
We need to justify why the right-hand side is in $K^-$. It is clearly in $K$. Suppose it is in $K^+$. By Proposition~\ref{prop:Kplus-2}, it can be expressed as $\sum_{a_1\in A_1} c_1(a_1) \cdot a_1 + \sum_{a_2 \in A_2} c_2(a_2) \cdot a_2$ for some nonnegative coefficients $c_1(a_1)$ and $c_2(a_2)$ with $\sum_{a_1 \in A_1} c_1 (a_1) + \sum_{a_2 \in A_2} c_2 (a_2) \ge 1$. This implies $v_1 = \sum_{a_1 \in A_1} c_1(a_1) \cdot a_1$ and $(1-t_1(v_1))v_2 = \sum_{a_2 \in A_2} c_2(a_2) \cdot a_2$. We then have $v_1 \in (\sum_{a_1 \in A_1} c_1(a_1)) \cdot K_1^+$, so $v_1 \notin (\sum_{a_1 \in A_1} c_1(a_1)) \cdot K_1^-$ and hence $t_1(v_1) \ge \sum_{a_1 \in A_1} c_1(a_1)$. This implies
\[
    1-t_1(v_1) \le 1-\sum_{a_1 \in A_1} c_1(a_1) \le \sum_{a_2 \in A_2} c_2(a_2)\,.
\]
But then $v_2 = \sum_{a_2 \in A_2} \frac{c_2(a_2)}{1-t_1(v_1)} \cdot a_2$ with $\sum_{a_2 \in A_2} \frac{c_2(a_2)}{1-t_1(v_1)} \ge 1$, so $v_2 \in K_2^+$.

Conversely, we can define a map
\[
    K^- \to K_1^- \times K_2^-\,,\quad v_1 \oplus v_2 \mapsto \left(v_1, \frac{v_2}{1-t_1(v_1)}\right)\,,
\]
which works because if $v_1 = \sum_{a_1 \in A_1} c_1(a_1) \cdot a_1$ with $\sum_{a_1 \in A_1} c_1(a_1) \ge 1$, then
\[
    t_1(v_1) \ge \sum_{a_1 \in A_1} c_1(a_1) \ge 1\,,
\]
a contradiction; and if $\frac{v_2}{1-t_1(v_1)} = \sum_{a_2 \in A_2} c_2(a_2) \cdot a_2$ with $\sum_{a_2 \in A_2} c_2(a_2) \ge 1$, then by choosing $c_1(a_1)$ such that $v_1 = \sum_{a_1 \in A_1} c_1(a_1) \cdot a_1$ and $\sum_{a_1 \in A_1} c_1(a_1) \ge t_1(v_1)$, we find that $v_1 \oplus v_2 = \sum_{a_1 \in A_1} c_1(a_1) \cdot a_1 + \sum_{a_2 \in A_2} (1-t_1(v_1)) c_2(a_2) \cdot a_2$ with
\[
    \sum_{a_1 \in A_1} c_1(a_1) + \sum_{a_2 \in A_2} (1-t_1(v_1))c_2(a_2) \ge t_1(v_1) + 1-t_1(v_1) = 1\,,
\]
another contradiction. These maps are inverse diffeomorphisms. It thus follows from the standard rules of volume calculation from calculus that the volume of $K^-$ is
\begin{align*}
    \vol(K^-) &= \int_{v_1 \in K_1^-} {\vol(K_2^-)}{(1-t_1(v_1))^{\dim(V_2)}} \cdot d^{\dim(V_1)} v_1 \\
    &= \int_0^1 \vol(K_1^-)\vol(K_2^-) (1-t)^{\dim(V_2)} \dim(V_1) t^{\dim(V_1)-1}\, dt \\
    &= \frac{\dim(V_1)!\dim(V_2)!}{(\dim(V_1)+\dim(V_2))!} \cdot \vol(K_1^-)\vol(K_2^-).
\end{align*}
Without loss of generality, the lattices used to measure the normalized volumes are all of the form $\bZ^n$ for some $n$, and we have $\Vol_\sbZ(K^-) = (\dim(V_1)+\dim(V_2))! \cdot \vol(K^-)$ and $\Vol_\sbZ(K_i^-) = \dim(V_i)! \cdot \vol(K_i^-)$ for each $i$, so the above equation amounts to $\Vol_\sbZ(K^-) = \Vol_\sbZ(K_1^-)\Vol_\sbZ(K_2^-)$.
\end{proof}

\begin{proposition}[{\cite[Chapter 6, proof of Theorem 2.3]{GKZ94}}]\label{prop:vol-cone-2}
Let $A \subseteq V\setminus \{0\}$ and $K, K^+$ be as in Proposition \ref{prop:Kplus-2}, and let $M_A$ be the lattice generated by the vectors of $A$. Let $Q$ be the convex hull of vectors of $A$. Suppose that the lattice $\tilde{M}_A$ affinely generated by elements of $A$ has codimension $1$ inside $M_A$. Then
\[
    \Vol_{M_A}(K\backslash K^+) = \Vol_{\tilde{M}_A}(Q)\,.
\]
\end{proposition}

\begin{lemma}\label{lem:general-prismatic-nonprismatic-u}
Suppose we have prismatic $F' = D'_1\cdots D'_{i-1} \boxed{C'_i}D'_{i+1} \cdots D'_n$ and nonprismatic $F = (\boxed{C_1}\cdots \boxed{C_n}, D_1 \cdots D_n)$ with $F' < F$. Then
\[
    u(F,F') = \binom{|D_i| - 2 + \sum_{j \neq i} |C_j| }{|D_i| - 1}
\]
\end{lemma}
\begin{proof}
As in Lemma \ref{lem:nonprismatic-volume-explicit}, let $W \coloneqq V/\bL(F')$ with projection map $p\colon V \to W$, and let $w' \coloneqq p(\phi(v))$ and $G = p(\ka_F)$. Let $K \subseteq W$ be the convex hull of the semigroup $S_G$ generated by $G$, where $w'$ is the additive identity, and let $K^+$ be the convex hull of $S_G \setminus \{w'\}$. However, unlike in Lemma \ref{lem:nonprismatic-volume-explicit}, we have $F'$ prismatic and $F$ nonprismatic rather than both being nonprismatic.

\noindent By Lemma \ref{lem:general-prismatic-lattice} applied to $F'$, in the space $W$,
\[
    \begin{cases}
        p(\phi(v_{j, \delta'_j})) = 0 & \text{for $j \in [n]$ and $\delta'_j \in D'_j$} \\
        p(\phi(v_{i,\gamma'_i})) = 0 & \text{for $\gamma'_i \in C_i$.}
    \end{cases}
\]
By a similar argument as in Lemma \ref{lem:general-nonprismatic-lattice}, $S_G$ is generated by
\[
    V' =
    \begin{cases}
        p(\phi(v_{j,\delta'_j}))  & \text{for $j \in [n] \setminus \{i\} $ and $\delta'_j \in D_j \setminus \{\delta_j\}$}, \\
        p(\phi(v_{i,\gamma'_i})) & \text{for $\gamma'_i \in C_i \setminus \{\gamma_i\}$} \\
        p(\phi(v_{j,\gamma_{j\,;\,i,\delta_i}})) & \text{for $j \in [n] \setminus \{i\}$, $\gamma_j \in C_j$, and $\delta_i \in D_i$.}
    \end{cases}
\]
To compute the volume of $K \setminus K^+$, let $V_1 \subseteq W$ be the vector space spanned by
\[
    A_1 \coloneqq \{p(\phi(v_{j,\delta'_j})) \mid \text{for $j \in [n] \setminus \{i\} $ and $\delta'_j \in D_j \setminus \{\delta_j\}$}\} \cup \{p(\phi(v_{i,\gamma'_i})) \mid \text{for $\gamma'_i \in C_i \setminus \{\gamma_i\}$}\}
\]
and let $V_2 \subseteq W$ be the vector space spanned by
\[
    A_2 \coloneqq \{p(\phi(v_{j, \gamma_{j\,;\,i,\delta_i}})) \mid \text{for $j \in [n] \setminus \{i\}, \gamma_j \in C_j, \delta_i \in D_i$}\}\,.
\]
Then $W = V_1 \oplus V_2$ and $V' = A_1 \sqcup A_2$. Letting $K_1, K_2, K_1^+, K_2^+$ be as in the statement of Proposition \ref{prop:vol-mult-2}, we deduce that $\Vol_\sbZ(K \setminus K^+) = \Vol_\sbZ(K_1 \setminus K_1^+) \Vol_\sbZ(K_2 \setminus K_2^+)$. Since the vectors of $A_1$ are linearly independent, we have $\Vol_\sbZ(K_1 \setminus K_1^+) = 1$ as usual.

To compute $\Vol_\sbZ(K_2 \setminus K_2^+)$, we apply Proposition \ref{prop:vol-cone-2}. Furthermore, we observe that each element $p(\phi(v_{j, \gamma_{j\,;\,i,\delta_i}}))$ of $A_2$ can be written as $p(\phi(v) + (e^j_{\gamma_j} - e^i_{\gamma_i}) + (f^i_{\delta_i} - f^j_{\delta_j}))$. Define
\begin{align*}
    \bfE &\coloneqq \{e^j_{\gamma_j} - e^i_{\gamma_i}- f^j_{\delta_j} \mid j \in [n] \setminus \{i\}, \gamma_i \in C_j\}\\
    \bfF &\coloneqq \{f^i_{\delta_i}  \mid  \gamma_i \in D_i\}\\   
    \bfG &\coloneqq \{(e^j_{\gamma_j} - e^i_{\gamma_i}) + (f^i_{\delta_i} - f^j_{\delta_j}) \mid j \in [n] \setminus \{i\}, \gamma_j \in C_j, \delta_i \in D_i\}\,.
\end{align*}
Then there is a bijection $\bfE\times\bfF\to\bfG$ given by $(a,b) \mapsto a + b$ which implies
\begin{align*}
    \Vol_\sbZ(\conv(\bfG)) &= \frac{(|\bfE| - 1 + |\bfF| - 1)!}{(|\bfE| - 1)! (|\bfF| - 1)!} \cdot \Vol_\sbZ(\conv(\bfE)) \cdot \Vol_\sbZ(\conv(\bfF)) \\
    &= \frac{\left(\left(\sum_{j \in [n] \setminus \{1\}} |C_j|\right) - 1 + |D_i| - 1\right)!}{\left(\left(\sum_{j \in [n] \setminus \{1\}} |C_j|\right) - 1\right)!(|D_i| - 1)!} \cdot 1 \cdot 1 \\
    &= \binom{|D_i| - 2 + \sum_{j\neq i}|C_j|}{|D_i| - 1}\,.\qedhere
\end{align*}
\end{proof}

While the function values of $u$ depend on whether the faces are prismatic or nonprismatic, the function $i$ always evaluates to $1$ for all types of faces of the Nashtope. This is shown in the following lemma.

\begin{lemma}\label{lem:general-i-equals-1}
For two Nashfaces $F,F'$ with $F' \le F$, we have $i(F,F') = 1$.
\end{lemma}
\begin{proof}
Let $M_F$, $M_{F'}$ be the affine $\bZ$-lattices generated by $F$ and $F'$.
We can use Lemmas \ref{lem:general-prismatic-lattice} and \ref{lem:general-nonprismatic-lattice} to find a basis for $M_{F'}$ that extends to a basis for $M_F$. Therefore, intersecting $M_F \cap \bL(F')$ gives $M_{F'}$. 
\end{proof}

\subsection{The Euler obstruction function for Nashfaces}\label{subsec: Euler obstruction}

We have now gathered all the ingredients to compute the Euler obstruction of faces of the Nashtope. Since it is defined recursively over the containment in larger faces, the proofs proceed by induction. We again need to distinguish between prismatic and nonprismatic faces.

\begin{lemma}\label{lem:nonprismatic-eu}
For any Nashface $F$, if $F$ is nonprismatic, then $\eu(F) = 1$.
\end{lemma}
\begin{proof}
We proceed by induction on $\codim(F)$. If $\codim(F) = 0$, which is equivalent to $F = P$, then by definition, $\eu(F) = 1$. Now, assume that all faces with codimension less than $c>0$ have Euler obstruction 1. We show that $\eu(F) = 1$ for all Nashfaces $F$ with $\codim(F) = c$. Let $F = (\boxed{C_1}\cdots \boxed{C_n}, D_1 \cdots D_n)$. Since $F$ is nonprismatic, any face containing it is also nonprismatic so that we can plug the identities of Lemmas \ref{lem:non-prismatic-u} and \ref{lem:general-i-equals-1} into Definition \ref{def: eu}:
\begin{align*}
    \eu(F) &= \sum_{F' > F} (-1)^{\dim(F') - \dim(F) - 1} \cdot u(F',F) \cdot i(F',F) \cdot \eu(F') \\
    &= \sum_{F' > F} (-1)^{\dim(F') - \dim(F) - 1} \cdot 1 \cdot 1 \cdot 1\\
    &= 1 - \sum_{F' \geq F} (-1)^{\dim(F') - \dim(F)}\,.
\end{align*}
Since
\[
    \sum_{F' \geq F} (-1)^{\dim(F') - \dim(F)} = \sum_{\substack{C'_i \subseteq [d_i - 1] \setminus C_i \\ D'_i \subseteq [d_i] \setminus D_i\\ i\in[n]}} (-1)^{\sum_{i=1}^n (|C'_i| + |D'_i|)} = (1 - 1)^{\sum_{i=1}^n (d_i - 1 + d_i - |C_i| - |D_i|)} = 0\,,
\]
we obtain $\eu(F) = 1$. The inductive step is complete. We conclude that all nonprismatic faces have Euler obstruction 1.
\end{proof}

\begin{lemma}\label{lem:prismatic-eu}
For any Nashface $F$, if $F$ is $i$-prismatic, then $\eu(F) = \min\left(d_i,\sum_{j \neq i}(d_j - 1)\right)$.
\end{lemma}
\begin{proof}
First, observe that the polytope $P$ is nonprismatic. Thus, from here on, we assume $F' \neq P$. Using the recursion for Euler obstructions, and breaking the sum into two pieces based on whether or not $F$ is prismatic, we write
\begin{align*}
    \eu(F) &= \sum_{\substack{F' > F \\\text{$F'$ nonprismatic}}} (-1)^{\dim(F') - \dim(F) - 1} \cdot u(F',F) \cdot i(F',F) \cdot \eu(F')\\
    &\quad+\sum_{\substack{F' > F \\\text{$F'$ prismatic}}} (-1)^{\dim(F') - \dim(F) - 1} \cdot u(F',F) \cdot i(F',F) \cdot \eu(F')\,.
\end{align*}
We simplify each summand. In the first summand, we can write $F' = (\boxed{C'_1} \cdots \boxed{C'_n}, D'_1 \cdots D'_n)$, substitute the value for $u$ calculated in Lemma \ref{lem:general-prismatic-nonprismatic-u}, the value for $i$ calculated in Lemma \ref{lem:general-i-equals-1}, and the value for $\eu$ calculated in Lemma \ref{lem:nonprismatic-eu} to get
\begin{align}
    &\sum_{\substack{F' > F \\\text{$F'$ nonprismatic}}} (-1)^{\dim(F') - \dim(F) - 1} \cdot u(F',F) \cdot i(F',F) \cdot \eu(F') \notag \\&= \sum_{\substack{F' > F \\\text{$F'$ nonprismatic}}} (-1)^{\dim(F') - \dim(F) - 1} \cdot \binom{|D'_i| - 2 + \sum_{j \neq i} |C'_j| }{|D'_i| - 1} \cdot 1 \cdot 1 \notag\\
    &= \sum_{r,s \ge 0} \binom{r + s}{r} \sum_{\substack{F' > F \\ |D'_i| - 1 = r \\ \sum_{j\neq i} |C_j'| - 1 = s}} (-1)^{\dim(F') - \dim(F) - 1}.\notag \\
    &= \sum_{r,s \ge 0} \binom{r + s}{r} \sum_{\substack{D_j \subseteq D'_j \subseteq [d_j]\text{ for }j \ne i,\ C_i \subseteq C'_i \subseteq [d_i-1],\\ D'_i \subseteq [d_i],\ C_j' \subseteq [d_j-1]\text{ for }j \ne i \\ |D'_i| - 1 = r,\ \sum_{j\neq i} |C_j'| - 1 = s}} (-1)^{\sum_{i=1}^n (|D'_i| + |C'_i|) - |C_i|-\sum_{j\ne i} |D_j|}\label{eq:eu-nonprismatic-gf}
\end{align}
Observe that if $D_j \neq [d_j]$ for some $j \neq i$, or if $C_i \neq [d_i - 1]$ (equivalently, when $F$ is not a maximal prism $[d_1] \cdots [d_{i-1}] \boxed{[d_i - 1]} [d_{i+1}] \cdots [d_n]$), then expression \ref{eq:eu-nonprismatic-gf} vanishes. In the maximal prism case,
\begin{align*}
    \eu(F) 
    &= \sum_{\substack{F' > F \\\text{$F'$ nonprismatic}}} (-1)^{\dim(F') - \dim(F) - 1} \cdot u(F',F) \cdot i(F',F) \cdot \eu(F') \\
    &= \sum_{r,s \ge 0} \binom{r + s}{r} \sum_{\substack{F' > F \\ |D'_i| - 1 = r \\ \sum_{j \neq i} |C_j'| - 1 = s}} (-1)^{\dim(F') - \dim(F) - 1} \\
    &= \sum_{r,s \ge 0} \binom{r + s}{r} \sum_{\substack{C'_j \subseteq [d_j - 1],\,j\neq i \\ D'_i \subseteq [d_i]\\ |D'_i| - 1 = r,\sum_{j \neq i} |C_j'| - 1 = s}} (-1)^{|D'_i| + \sum_{j\neq i}|C_j'| - 2} \\
    &= \sum_{r,s \ge 0} \binom{r + s}{r} (-1)^{r+s} \binom{d_i}{r+1} \binom{\sum_{j\neq i}(d_j - 1)}{s+1}\,. 
\end{align*}
Define
\[
    c_{a,b} \coloneqq \sum_{r,s \ge 0} \binom{r + s}{r} (-1)^{r+s} \binom{a}{r+1} \binom{b}{s+1}\,.
\]
We claim that $c_{a,b} = \min(a,b)$. Indeed write
\begin{align*}
    \sum_{a,b \ge 0} c_{a,b}\,x^ay^b &= \sum_{r,s \ge 0} \binom{r+s}{r} (-1)^{r+s}\cdot\sum_{a \ge 0} \binom{a}{r+1} x^a\cdot\sum_{b \ge 0} \binom{b}{s + 1} y^b \\
    &= \sum_{r,s \ge 0} \binom{r +s}{r} (-1)^{r+s} \cdot\frac{x^{r+1}}{(1 - x)^{r+2}}\cdot\frac{y^{s+1}}{(1 - y)^{s+2}} \\
    &= \frac{x}{(1 - x)^2}\cdot\frac{y}{(1 - y)^2}\cdot \sum_{r,s \ge 0} \binom{r+s}{r} \left(-\frac{x}{1 - x}\right)^r\left(-\frac{y}{1 - y}\right)^s \\
    &= \frac{x}{(1 - x)^2}\cdot\frac{y}{(1 - y)^2}\cdot\left(1 + \frac{x}{1 -x} + \frac{y}{1 - y}\right)^{-1} \\
    &= \frac{xy}{(1 - x)(1 - y)(1 - xy)} \\
    &= \sum_{j,k \ge 1} x^j y^k\cdot\sum_{i \ge 0}(xy)^i \\
    &= \sum_{a,b \ge 0} x^ay^b \cdot \#\{i \ge 0: a - i \ge 1, b - i \ge 1\} \\
    &= \sum_{a,b \ge 0} \min(a,b)\,x^ay^b\,.
\end{align*}
So 
\[
    \eu(F) = \sum_{r,s \ge 0} \binom{r + s}{r} (-1)^{r+s} \binom{d_i}{r+1} \binom{\sum_{j\neq i}(d_j - 1)}{s+1} = \min\left(d_i,\sum_{j \neq i}(d_j - 1)\right)\,.
\]
In the non-maximal-prism case, the first summand vanishes, so we have the following equality.
\[
    \eu(F) =  \sum_{\substack{F' > F \\\text{$F'$ prismatic}}} (-1)^{\dim(F') - \dim(F) - 1} \cdot u(F',F) \cdot i(F',F) \cdot \eu(F')\,
\]
We can then use an inductive argument similar to that of Lemma \ref{lem:nonprismatic-eu} to conclude that $\eu(F) = \eu(F'')$, where $F''$ is the unique maximal prism containing $F$. 
\end{proof}

\begin{remark}
From Lemma \ref{lem:prismatic-eu}, we see why the computation of the degree of $\nabla_{\!\ka}$ differs between strictly balanced and on-boundary game formats. For a strictly balanced game format, the lemma gives $\eu(F) = d_n$ for $n$-prismatic $F$, whereas for an on-boundary game formats, the lemma gives $\eu(F) = \sum_{i=1}^{n-1} (d_i - 1)$ for $n$-prismatic $F$.
\end{remark}

\subsection{The normalized volumes of Nashfaces}\label{subsec: normalized volumes}

The remaining indeterminate in the degree formula in Theorem~\ref{thm:polytope} is the normalized volume of the faces of the Nashtope. We construct a function on sequences of boxed and unboxed integers that resembles the actual volume when these sequences coincide with face strings, but provides a generalization of such which will prove useful in computations of the degree in the following sections.

\begin{lemma}\label{lem:face-volume-2}
Let $F$ be a non-vertex face of the Nashtope, and fix a vertex $v$ of $F$. Then
\[
    \Vol_\sbZ(F) = \sum_{\substack{v \not\in F' < F \\ \dim(F)-\dim(F')=1}} \Vol_\sbZ(F')
\]
where the sum ranges over faces $F'$ of $F$ with codimension $1$ in $F$ so that $F'$ does not contain $v$.
\end{lemma}
\begin{proof}
Let $\kf$ be the set of $F'$ described in the lemma statement. For each $F' \in \kf$, let $\tilde{F}'$ be the convex hull of $F' \cup \{v\}$. Since $F$ is convex, $F = \cup_{F' \in \kf} \tilde{F}'$. For any distinct $F',F''$ in $\kf$, we have that $\tilde{F}' \cap \tilde{F}''$ is measure zero inside $\bL(F)$. Thus,
\[
    \Vol_\sbZ(F) = \sum_{F' \in \kf} \Vol_\sbZ(\tilde{F}')\,.
\]
Fix $F' \in \kf$ and let $M_F,M_{F'}$ be the lattices generated by $F,F'$ respectively. By Lemmas \ref{lem:general-prismatic-lattice} and \ref{lem:general-nonprismatic-lattice}, we can choose a basis set for $M_F$ so that $M_F = \bZ^m$ and a basis set for $M_{F'}$ so that $M_{F'} = 1 \oplus \bZ^{m-1}$ for some $m$. Finally, apply Proposition \ref{prop:vol-cone-2} to obtain the desired result.
\end{proof}

\begin{definition}
Let $a_1,\dots, a_n, b_1 \dots, b_n$ be nonnegative integers.
\begin{enumerate}
    \item If $|\{i\in[n]\mid a_i\ge 1\}|\ge 2$, define
    \[
        \Vol_\sbZ\left(\boxed{a_1}\dots\boxed{a_n}, b_1\cdots b_n\right)
    \]
    to be the volume of an arbitrary nonprismatic face $\boxed{C_1}\cdots \boxed{C_n}, D_1\cdots D_n$ with $|C_i| = a_i$ and $|D_i| = b_i + 1$ for all $i$.
    \item If $\{i\in[n]\mid a_i\ge 1\}=\{j\}$, define
    \[
        \Vol_\sbZ\left(\boxed{0}\cdots \setlength{\arraycolsep}{2pt}\begin{array}{|c|c|c|}\hline 0 & a_j & 0 \\\hline\end{array} \cdots \boxed{0}, b_1 \cdots b_n\right)
    \]
    to be the volume of an arbitrary prism $D_1 \cdots D_{j-1}\boxed{C_j}D_{j+1}\cdots D_n$ with $|C_j| = a_j$ and $|D_\ell| = b_\ell + 1$, for all $\ell \neq j$, if $b_j = 0$ and $0$ otherwise.
    \item Finally, if $a_i=0$ for all $i\in[n]$, define
    \[
        \Vol_\sbZ\left(\boxed{0} \cdots \boxed{0}, b_1 \cdots b_n\right) =
        \begin{cases}
            1 & \text{if $b_i=0$ for all $i\in[n]$}\\
            0 & \text{otherwise.}
        \end{cases}
    \]
\end{enumerate}
\end{definition}

\begin{corollary}\label{cor:vol-format}
Let $a_1,\dots, a_n, b_1 \dots, b_n$ be nonnegative integers with $a_i \ge 1$ for some $i\in[n]$. Then
\[
\begin{aligned}
    \Vol_\sbZ\left(
    \setlength{\arraycolsep}{2pt}\begin{array}{|c|}\hline a_1\\\hline\end{array}
    \cdots
    \setlength{\arraycolsep}{2pt}\begin{array}{|c|}\hline a_n\\\hline\end{array},
    b_1\cdots b_n\right) = \Vol_\sbZ\left(
    \setlength{\arraycolsep}{2pt}\begin{array}{|c|}\hline a_1\\\hline\end{array}
    \cdots
    \setlength{\arraycolsep}{2pt}\begin{array}{|c|c|c|}\hline a_{i-1} & a_i - 1 & a_{i+1} \\\hline\end{array}
    \cdots
    \setlength{\arraycolsep}{2pt}\begin{array}{|c|}\hline a_n\\\hline\end{array},
    b_1 \cdots b_n\right)\\
    + \sum_{\substack{j\neq i\\ b_j\ge 1}} \Vol_\sbZ\left(
    \setlength{\arraycolsep}{2pt}\begin{array}{|c|}\hline a_1\\\hline\end{array}
    \cdots
    \setlength{\arraycolsep}{2pt}\begin{array}{|c|c|c|}\hline a_{j-1} & a_j & a_{j+1} \\\hline\end{array}
    \cdots
    \setlength{\arraycolsep}{2pt}\begin{array}{|c|}\hline a_n\\\hline\end{array},
    b_1 \cdots b_{j-1}(b_j-1)b_{j+1}\cdots b_n\right)\,.
\end{aligned}
\]
\end{corollary}
\begin{proof}
This follows directly from Lemma \ref{lem:face-volume-2}.
\end{proof}

For a vector $\bv \coloneqq (v_1, \dots, v_n)$, let $v_1, \dots, v_n \ge 0$ be abbreviated as $\bv \ge 0$. For variables $q_1, \dots, q_n$ and $r_1, \dots, r_n$, the expression $\prod_{i=1}^n q_i^{r_i}$ is abbreviated $\bq^\br$. Given $\bx=(x_1,\ldots,x_n)$, we also adopt the notation $\hat{x}_i\coloneqq\sum_{j\neq i}x_j$.

\begin{lemma}\label{lem:vol-gf}
For any $a_1,\dots,a_n \ge 0$, we have
\[
    \sum_{\bb\ge 0} \Vol_\sbZ\left(\boxed{a_1} \cdots \boxed{a_n}, b_1 \cdots b_n\right)\,\bx^\bb = \prod_{i=1}^n\frac{1}{(1-\hat{x}_i)^{a_i}}\,.
\]
\end{lemma}
\begin{proof}
    This lemma is best understood combinatorially. Consider all the ways we can repeatedly choose terms on the right side of the recursive formula until all the $a_i$ are zero. Let $\alpha_1,\dots, \alpha_n$ and $\beta_1,\dots,\beta_n$ be the operations of decrementing $a_i$ or $b_i$. Assume that we always choose the leftmost $a_i \ge 1$ to apply the formula. In total, every valid path is a string consisting of $a_i$ $\alpha_i$s and $b_i$ $\beta_i$s for each $i$, with the $\alpha_i$s appearing in order. The last letter must be an $\alpha_i$. Call the strings to the left of each $\alpha_i$ ``chunks.'' Then every $\beta_j$ except for $\beta_i$ can appear in the chunk to the left of each $\alpha_i$.
\end{proof}

\subsection{The degree of the \texorpdfstring{$\ka$}{A}-discriminant for strictly balanced game formats}\label{subsec: degree A-discriminant strictly balanced} Throughout this section, we assume $d_n - 1 < \sum_{i=1}^{n-1}(d_i - 1)$. We compute $\deg(\nabla_{\!\ka})$ and use the following notation throughout this section and in Section~\ref{subsec: degree A-discriminant on-boundary}:
\begin{enumerate}
    \item Given a vector $\bv=(v_1,\ldots,v_n)$, let $e_1(\bv) = \sum_{i=1}^n v_i$, $e_2(\bv) = \sum_{1 \le i < j \le n} v_iv_j$, and $e_3(\bv) =  \sum_{1 \le i < j < k \le n} v_iv_jv_k$ be the first, second, and third elementary symmetric functions in the entries of $\bv$.
    \item The vector $(v_1-1,\ldots,v_n-1)$ is abbreviated as $\bv-\bone$.
    \item $C \coloneqq \sum_{i=1}^n (d_i - 1)$. 
\end{enumerate}

In the upcoming statements, we utilize formal variables $x_1,\ldots,x_n$ and $y_1,\ldots,y_n$. Define $z_i \coloneqq \frac{y_i}{1 + y_i}$ and $\hat{z}_i\coloneqq \sum_{j\neq i}z_j$ for all $i\in[n]$. For use starting in Lemma \ref{lem:delta-2-boundary}, we introduce the quantities
\begin{align*}
    \alpha_i \coloneqq \left(1 - \frac{x_i \hat{z}_i}{1 - \hat{z}_i}\right)^{\!\!-1}\,,\quad\beta_i \coloneqq \frac{x_i}{(1 - \hat{z}_i)(1 - (1 + x_i)\hat{z}_i)}\,,\quad\gamma_i \coloneqq \frac{x_i(x_i + (1 + x_i)\hat{z}_i)}{(1 - \hat{z}_i)(1 - (1 + x_i)\hat{z}_i)^2}
\end{align*}
for all $i\in[n]$, and the vectors $\alpha=(\alpha_1,\ldots,\alpha_n)$, $\beta=(\beta_1,\ldots,\beta_n)$, $\gamma=(\gamma_1,\ldots,\gamma_n)$.
Finally, for ease of readability, we define the quantities
\[
    Y \coloneqq \prod_{i=1}^n\frac{1}{(1 + y_i)^2}\,,\quad Y_i \coloneqq \prod_{j \neq i} \frac{1}{1 - x_jy_i}\,.
\]

\begin{lemma}\label{lem:delta-1-strictly-balanced}
Let $x_1,\dots,x_n$ and $y_1,\dots,y_n$ be formal variables. The expression $\delta_1$ in \eqref{eq: delta1} equals the coefficient of $(x_1y_1)^{d_1 - 1} \cdots (x_ny_n)^{d_n - 1}$ in the formal power series
\begin{equation}\label{eq:delta-1-lemma-statement-strictly-balanced}
    \prod_{i=1}^n \frac{1 - \hat{z}_i}{(1 + y_i)^2(1 - (1 + x_i)\hat{z}_i)}\cdot\sum_{i=1}^n \frac{x_i}{(1 - \hat{z}_i)(1 - (1 + x_i)\hat{z}_i)}\,.
\end{equation}
\end{lemma}
\begin{proof}
Recall that
\[
    \delta_1 = \sum_{F \le P} (-1)^{\codim(F)}(\dim(F) + 1)\cdot \Vol_\sbZ(F) \cdot \eu(F)\,.
\]
Write $V(\ba,\bb)=\Vol_\sbZ\left(\boxed{a_1} \cdots \boxed{a_n}, b_1 \cdots b_n\right)$ for simplicity. Define
\begin{equation}\label{eq: expansion delta-1}
    \hat{\delta}_1(\bc, \bd') \coloneqq \sum_{\ba,\bb\ge 0} \prod_{i=1}^n\binom{c_i}{a_i}\binom{d'_i}{b_i+1} (-1)^{e_1(\bc+\bd'-\ba-\bb)}e_1(\ba+\bb)\,V(\ba,\bb)
\end{equation}
and observe that by Lemmas \ref{lem:prismatic-eu} and \ref{lem:nonprismatic-eu} for the Euler obstructions, and Lemma \ref{lem:vol-gf} for the face volumes, $\hat{\delta}_1(\bd-\bone,\bd) = \delta_1$.\footnote{Note that, in the on-boundary case, a correction term is needed. Compare with the proof of Lemma \ref{lem:delta-2-boundary}.}
Packaging this in a generating function, we write
\begin{align}\label{eq:delta-1-genfunc-strictly-balanced}
\begin{split}
    \sum_{\bc,\bd'\ge 0}\hat{\delta}_1(\bc, \bd')\,\bx^\bc \by^{\bd'} &= \sum_{\ba,\bb\ge 0} \prod_{i=1}^n \frac{x_i^{a_i}}{(1 + x_i)^{1 + a_i}} \cdot \frac{y_i^{b_i + 1}}{(1 + y_i)^{2 + b_i}}\,e_1(\ba+\bb)\,V(\ba,\bb)\\
    &\hspace{-2cm}= \prod_{i=1}^n\frac{y_i}{(1+x_i)(1+y_i)^2}\cdot \underbrace{\sum_{\ba,\bb\ge 0} \prod_{i=1}^n\left(\frac{x_i}{1 + x_i}\right)^{\!\!a_i}\left(\frac{y_i}{1 + y_i}\right)^{\!\!b_i}e_1(\ba+\bb)\,V(\ba,\bb)}_\Lambda\,.
\end{split}
\end{align}
From the identity $\sum_{\bb\ge 0}V(\ba,\bb)\,\bx^\bb=\prod_{i=1}^n(1-\hat{x}_i)^{-a_i}\eqqcolon F(\bx)$ we derive
\begin{align*}
    \sum_{\bb\ge 0}e_1(\bb)\,V(\ba,\bb)\,\bx^\bb
    &= \sum_{j=1}^n x_j \frac{\partial F}{\partial x_j}(\bx) = F(\bx)\sum_{j=1}^n x_j\sum_{i\neq j}\frac{a_i}{1-\hat{x}_i} = F(\bx)\sum_{i=1}^n\frac{a_i\,\hat{x}_i}{1-\hat{x}_i}\,.
\end{align*}
In particular
\[
    \sum_{\bb\ge 0}e_1(\ba+\bb)\,V(\ba,\bb)\,\bx^\bb = F(\bx)\sum_{i=1}^n \frac{a_i}{1-\hat{x}_i}\,.
\]
Now let $t_i\coloneqq\frac{x_i}{1+x_i}$ and $z_i\coloneqq\frac{y_i}{1+y_i}$ for all $i\in[n]$. Write $\bt=(t_1,\ldots,t_n)$ and $\bz=(z_1,\ldots,z_n)$. The sum $\Lambda$ in \eqref{eq:delta-1-genfunc-strictly-balanced} simplifies to
\begin{align*}
    \Lambda &= \sum_{\ba\ge 0} \bt^\ba\sum_{\bb\ge 0}e_1(\ba+\bb)\,V(\ba,\bb)\,\bz^\bb\\
    &= \sum_{\ba\ge 0} \bt^\ba\,F(\bt)\sum_{i=1}^n \frac{a_i}{1-\hat{z}_i}\\
    &= \sum_{j=1}^n \frac{1}{1-\hat{z}_j}\sum_{\ba \ge 0} a_j\,\left(\frac{t_j}{1-\hat{z}_j}\right)^{\!\!a_j}\prod_{i \neq j} \left(\frac{t_i}{1-\hat{z}_i}\right)^{\!\!a_i} \\
    &= \sum_{j=1}^n \frac{1}{1-\hat{z}_j}\cdot\sum_{a \ge 0}a\,\left(\frac{t_j}{1-\hat{z}_j}\right)^{\!\!a}\cdot\prod_{i \neq j}\sum_{a \ge 0}\left(\frac{t_i}{1-\hat{z}_i}\right)^{\!\!a} \\
    &= \sum_{j=1}^n \frac{1}{1-\hat{z}_j} \cdot \frac{t_j}{1-\hat{z}_j}\left(1 - \frac{t_j}{1-\hat{z}_j}\right)^{\!\!-2}\cdot\prod_{i \neq j}\left(1-\frac{t_i}{1-\hat{z}_i}\right)^{\!\!-1} \\
    &= \prod_{i=1}^n \left(1 - \frac{t_i}{1-\hat{z}_i}\right)^{\!\!-1}\cdot\sum_{i=1}^n \frac{t_i}{(1-\hat{z}_i)^2}\left(1 - \frac{t_i}{1-\hat{z}_i}\right)^{\!\!-1}\,.
\end{align*}
Plugging the last value of $\Lambda$ into \eqref{eq:delta-1-genfunc-strictly-balanced} and simplifying yields the desired statement.
\end{proof}

\noindent We now present the following technical lemma, which we use in the proofs of the next statements.

\begin{lemma}\label{lem:claim-3}
Let $c_1,\dots, c_n$ be nonnegative integers. Let $F(\ba) = F(a_1, \dots, a_n) = \sum_{\bm\in\bN^n} f_\bm\,\ba^\bm$. The coefficient of $\by^\bc$ in 
\begin{equation}\label{eq: expr 1}
    \prod_{j=1}^n \frac{1}{(1 + y_j)^2}\cdot F\bigg(\frac{y_1}{1 + y_1},\dots,\frac{y_n}{1 + y_n}\bigg)
\end{equation}
equals the coefficient of $\bz^\bc$ in
\begin{equation}\label{eq: expr 2}
    \prod_{j=1}^n(1 - z_j)^{c_j + 1}F(z_1,\dots,z_n)
\end{equation}
\end{lemma}
\begin{proof}
Expanding $F$ in \eqref{eq: expr 1} we get 
\begin{equation}\label{eq: expr 3}
    \sum_{\bm\in\bN^n} f_\bm \prod_{j=1}^n y_j^{m_j} (1 + y_j)^{-m_j-2}\,.
\end{equation}
Note that, for $\alpha>0$, the coefficient of $t^i$ in $(1+t)^{-\alpha}$ is $(-1)^i\binom{\alpha+i-1}{i}$. Hence the coefficient of $\by^\bc$ in \eqref{eq: expr 3} is
\begin{equation}\label{eq: expr 4}
    \sum_{\bm \le c} f_\bm \prod_{j=1}^n(-1)^{c_j-m_j}\binom{c_j+1}{c_j-m_j}\,.
\end{equation}
Comparing with the coefficient of $\bz^\bc$ in \eqref{eq: expr 2}, that is
\[
    \sum_{\bm \le c} f_\bm \prod_{j=1}^n (1-z_j)^{c_j+1}\cdot\left(\text{coeff. of $z_j^{c_j-m_j}$}\right)\,,
\]
we obtain again \eqref{eq: expr 4}. This concludes the proof.
\end{proof}

\begin{theorem}\label{thm:delta-1-strictly-balanced}
Let $z_1,\dots,z_n$ be formal variables. Then $\delta_1$ equals the coefficient of $\prod_{i=1}^nz_i^{d_i - 1}$ in 
\[
    \left(C + \sum_{i=1}^n(C - 2d_i + 1)z_i\right) \prod_{i=1}^n\left(1 + \hat{z}_j\right)^{d_i - 1}\,,\quad \hat{z}_i \coloneqq \sum_{j \neq i} z_j\quad\forall\,i\in[n]\,.
\]
\end{theorem}
\begin{proof}
We first eliminate the variables $x_1,\dots,x_n$ in Lemma \ref{lem:delta-1-strictly-balanced}. Isolating the $i$-th summand in Lemma \ref{lem:delta-1-strictly-balanced} and cancelling $1-\hat{z}_i$ gives
\[
    \left(\prod_{k \neq i} \frac{1 - \hat{z}_k}{(1 + y_k)^2(1 - (1 + x_k)\hat{z}_k)}\right) \frac{x_i}{(1 + y_i)^2(1 - (1 + x_i)\hat{z}_i)^2}\,.
\]
We get that, for $c_i \ge 1$, the coefficient of $x_i^{c_i}$ in $\frac{x_i}{(1 - (1 + x_i)\hat{z}_i)^2}$ is $\frac{c_i \hat{z}_i^{c_i - 1}}{(1 - \hat{z}_i)^{c_i + 1}}$.
Similarly, for $k \neq i$, the coefficient of $x_k^{c_k}$ in $\frac{1 - \hat{z}_k}{1 - (1 + x_k)\hat{z}_k}$ is $\frac{\hat{z}_k^{c_k}}{(1 - \hat{z}_k)^{c_k}}$.
Hence the desired degree $\delta_1$ is the coefficient of $\by^{\bd-\bone}=y_1^{d_1 - 1} \cdots y_n^{d_n - 1}$ in
\begin{equation}\label{eq:3parts}
    \prod_{j=1}^n \frac{1}{(1 + y_j)^2}\bigg(\sum_{i=1}^n \frac{(d_i - 1) \hat{z}_i^{d_i - 2}}{(1 - \hat{z}_i)^{d_i}}\bigg)\bigg( \prod_{k\neq i} \frac{\hat{z}_k^{d_k -1}}{(1 - \hat{z}_k)^{d_k - 1}}\bigg)\,.
\end{equation}
By Lemma \ref{lem:claim-3}, the coefficient of $\by^{\bd-\bone}$ in 
\[
    \prod_{j=1}^n \frac{1}{(1 + y_j)^2}\cdot F\left(\frac{y_1}{1 + y_1},\dots,\frac{y_n}{1 + y_n}\right)
\]
equals the coefficient of $\bz^{\bd-\bone}$ in
\[
    \prod_{j=1}^n(1 - z_j)^{d_j}F(z_1,\dots,z_n)
\]
where $F(\ba) = F(a_1, \dots, a_n) = \sum_{\bm \in \bN^n} f_\bm\,\ba^\bm$. 
We now rewrite the expression in \eqref{eq:3parts} in terms of the $z_i$'s, but first we need to set up some notation. Define 
\[
    Q_i \coloneqq (d_i - 1) \hat{z}_i^{d_i - 2} \prod_{k \neq i} \hat{z}_k^{d_k - 1}\,,\quad
    R_i \coloneqq \frac{\prod_{j=1}^n(1 - z_j)^{d_j}}{(1 - \hat{z}_i)^{d_i} \prod_{k\neq i} (1 - \hat{z}_k^{d_k - 1})}\,.
\]
Rewriting, we see that $\delta_1$ is the $\bz^{\bd-\bone}$ coefficient in $\sum_{i=1}^n Q_i R_i$. We now give first-order approximations of the numerator and denominator of $R_i$. For the numerator, we get 
\[
    \prod_{j=1}^n (1 - z_j)^{d_j} = 1 - \sum_{j=1}^n (d_j)z_j\mod{\langle z_1,\ldots,z_n\rangle^2}
\]
and for the denominator we get
\[
    \frac{1}{(1 - \hat{z}_i)^{d_j}}\prod_{k \neq i} \frac{1}{(1 - \hat{z}_k)^{d_k - 1}} = 1 + \hat{z}_i + \sum_{k=1}^n (d_k - 1)\hat{z}_k\mod{\langle z_1,\ldots,z_n\rangle^2}\,.
\]
Hence the first-order approximation of $R_i$ is
\[
    R_i\mod{\langle z_1,\ldots,z_n\rangle^2} = 1 + \sum_{k=1}^n (d_k - 1) \hat{z}_k - \sum_{j=1}^n d_jz_j + \hat{z}_i = 1 + \sum_{j=1}^n (C - 2d_j + 1)z_j + \hat{z}_i\,.
\]
Since $\deg(Q_i) = C - 1$, the constant term $1$ in $R_i$ does not contribute to the $\bz^{\bd-\bone}$ term. Thus $\delta_1$ is the coefficient of $\bz^{\bd-\bone}$ in 
\begin{align*}
    \sum_{i=1}^n Q_i\left[\sum_{j=1}^n (C - 2d_j + 1)z_j + \hat{z}_i\right]
    &= \sum_{i=1}^n (d_i - 1) \hat{z}_i^{d_i - 2} \prod_{k \neq i} \hat{z}_k^{d_k - 1}\left[\sum_{j=1}^n (C - 2d_j + 1)z_j + \hat{z}_i\right]\\
    &\hspace{-1.2cm}= \bigg(\sum_{i=1}^n (d_i - 1) \hat{z}_i^{d_i - 2} \prod_{k \neq i}\hat{z}_k^{d_k - 1}\bigg)\bigg( \sum_{j=1}^n (C - 2d_j + 1)z_j \bigg) + C\prod_{i=1}^n \hat{z}_i^{d_i - 1}\,.
\end{align*}
Finally, observe that the degree-$C$ homogeneous part of $\prod_{i=1}^n (1 + \hat{z}_i)^{d_i - 1}$ is $\prod_{i=1}^n \hat{z}_i^{d_i - 1}$ and the degree-$(C-1)$ part is $\sum_{i=1}^n Q_i$. We thus obtain the desired formula.
\end{proof}

\begin{corollary}\label{cor:discriminant-strictly-balanced}
Let $\bd$ be a strictly balanced game format. The $\ka$-discriminant $\nabla_{\!\ka}$ is a hypersurface. Its degree is the coefficient of $\prod_{i=1}^n z_i^{d_i-1}$ in 
\begin{equation}\label{eq:strictly-balanced}
    \left(C + \sum_{i=1}^n(C - 2d_i + 1)z_i\right) \prod_{i=1}^n(1 + \hat{z}_i)^{d_i-1}\,.
\end{equation}
\end{corollary}
\begin{proof}
Since $C - 2d_i + 1\ge 0$ for every $i \in [n]$, the expression \eqref{eq:strictly-balanced} cannot have negative coefficients. It suffices to identify one positive contribution to the coefficient of $\prod_{i=1}^n z_i^{d_i-1}$. Considering the product $\prod_{i=1}^n(1 + \hat{z}_i)^{d_i - 1}$, we can see that we must choose a linear term from each factor, so it suffices to show that there exists a positive contribution to the coefficient of $\prod_{i=1}^n z_i^{d_i-1}$  from $\prod_{i=1}^n\hat{z}_i^{d_i-1}$. This is done in \cite[Theorem 2.8]{abo2026vector}.
\end{proof}

Although Corollary \ref{cor:discriminant-strictly-balanced} does not give a closed-form degree formula for $\nabla_{\!\ka}$, special values of $\bd$ allow a closed form from Corollary \ref{cor:discriminant-strictly-balanced}. 
The following result sharpens the degree bound in \cite[Example~5.6]{pratt2026multigraded}.

\begin{lemma}\label{lem:binary-derangement}
For a binary game, i.e.\, if $d_i = 2$ for all $i \in [n]$,
\begin{equation}\label{eq: deg A-discriminant binary}
    \deg(\nabla_{\!\ka}) = n\,[(n - 2)\,\cdot\,!n + (n - 3)\,\cdot\,!(n - 1)]\,,
\end{equation}
where $!n$ is the number of derangements of $[n]$.
\end{lemma}
\begin{proof}
We substitute $d_i = 2 $ for all $i \in [n]$ and $C = n$ into the formula in Corollary \ref{cor:discriminant-strictly-balanced} to get
\[
    \left(n + \sum_{i=1}^n (n - 3)z_i\right)\prod_{i=1}^n (1 + \hat{z}_i) = n\prod_{i=1}^n (1 + \hat{z}_i) + (n - 3)\cdot\sum_{i=1}^n z_i\cdot\prod_{i=1}^n (1 + \hat{z}_i).
\]
There is a bijective map $\phi$ between derangements of $n$ and positive contributions to the $\prod_{i=1}^n z_i$ coefficient of $\prod_{i=1}^n (1 + \hat{z}_i)$. Thus the term $n\prod_{i=1}^n (1 + \hat{z}_i)$ contributes $n\,\cdot\,!n$ to $\deg(\nabla_{\!\ka})$. Now consider $\sum_{i=1}^n z_i\cdot\prod_{i=1}^n (1 + \hat{z}_i)$. Each positive contribution to the coefficient of $\prod_{i=1}^n z_i$ selects one linear term from $\sum_{i=1}^n z_i$ and $n - 1$ linear terms from $\prod_{i=1}^n (1 + \hat{z}_i)$.
We can choose $n$ linear terms from $\prod_{i=1}^n (1 + \hat{z}_i)$ and select one such linear term to replace with the corresponding linear term of $\sum_{i=1}^n z_i$, and there are $n\,\cdot\,!n$ ways to do this, or the chosen linear term from $\sum_{i=1}^n z_i$ cannot be replaced with a corresponding linear term from the $\prod_{i=1}^n (1 + \hat{z}_i)$, and there are $n\,\cdot\,!(n - 1)$ such cases.
Accounting for the factor of $n-3$, we see that the total contribution of the term $(n-3)\cdot\sum_{i=1}^n z_i\cdot\prod_{i=1}^n (1 + \hat{z}_i)$ to the degree of $\nabla_{\!\ka}$ is $(n-3)[n\,\cdot\,!n) + n\,\cdot\,!(n - 1)]$. Adding these two contributions to the degree of $\nabla_{\!\ka}$ gives
\[
    n\,\cdot\,!n + (n-3)[n\,\cdot\,!n + n\,\cdot\,!(n-1)] = n\,[!n + (n-3)\,\cdot\,!n + (n-3)\,\cdot\,!(n-1)]\,,
\]
which coincides with \eqref{eq: deg A-discriminant binary}.
\end{proof}

\subsection{The degree of the \texorpdfstring{$\ka$}{A}-discriminant for on-boundary game formats}\label{subsec: degree A-discriminant on-boundary}
Throughout this section, we assume  $d_n - 1 = \sum_{i=1}^{n-1}(d_i - 1)$. 
We now compute $\deg(\nabla_{\!\ka})$.

\begin{lemma}\label{lem:delta-1-boundary}
If $d_n - 1 = \sum_{i=1}^{n-1}(d_i - 1)$, then  $\delta_1 = 0$.
\end{lemma}
\begin{proof}
In the on-boundary case, the expression for $\delta_1$ is obtained by subtracting
\[
\resizebox{\textwidth}{!}{
    $\begin{aligned}
    \sum_{\substack{a_n \ge 0 \\ b_i\ge 0\\ i\in[n-1]}} \binom{d_n - 1}{a_n} \prod_{i=1}^{n-1}\binom{d_i}{b_i + 1} (-1)^{d_n - 1 - a_n + \sum_{i=1}^{n-1}(d_i - b_i - 1)}\left(a_n + \sum_{i=1}^{n-1} b_i\right)\Vol_\sbZ\left(0\cdots 0\,\boxed{a_n}, b_1\cdots b_{n-1}\,0\right)
    \end{aligned}$
    }
\]
from \eqref{eq:delta-1-genfunc-strictly-balanced}.
An argument very similar to the proof of Lemma \ref{lem:delta-1-strictly-balanced} shows that $\delta_1$ is the coefficient of $\prod_{i=1}^n(x_iy_i)^{d_i - 1}$ in \eqref{eq:delta-1-lemma-statement-strictly-balanced} minus the coefficient of $x_n^{d_n - 1} \prod_{j=1}^{n-1}y_j^{d_j - 1}$ in \eqref{eq:delta-1-lemma-statement-strictly-balanced} evaluated at $x_j = 0$ for $j \in [n-1]$ and $y_n = 0$. This evaluation gives
\[
    \frac{x_n}{(1 - (1 + x_n)\hat{z}_n)^2 } \prod_{j=1}^{n-1} \frac{1}{(1 + y_j)^2}\,.
\]
This is the same as the coefficient of $\prod_{j=1}^{n-1} y_j^{d_j - 1}$ in 
\[
    (d_n - 1) \hat{z}_n^{d_n - 2}(1 - \hat{z}_n)^{-d_n} \prod_{j=1}^{n-1} \frac{1}{(1 + y_j)^2}\,.
\]
By the part of Theorem \ref{thm:delta-1-strictly-balanced} used to derive \eqref{eq:3parts}, it suffices to show that the aforementioned coefficient equals the coefficient of $\prod_{j=1}^ny_j^{d_j - 1}$ in 
\[
    \prod_{j=1}^n \frac{\hat{z}_j^{d_j - 1}}{(1 + y_j)^2(1 - \hat{z}_j)^{d_j - 1}}\cdot\sum_{j=1}^n \frac{d_j - 1}{\hat{z}_j(1 - \hat{z}_j)}\,.
\]
Because of the on-boundary condition, monomials containing $y_n^{d_n-1}$ can only appear in the summand corresponding to $j = n$, namely
\[
    (d_n - 1) \hat{z}_n^{d_n - 2}(1 - \hat{z}_n)^{-d_n}(1 + y_n)^{-2} \prod_{j=1}^{n-1} \frac{\hat{z}_j^{d_j - 1}}{(1+ y_j)^2(1 - \hat{z}_j)^{d_j - 1}}\,.
\]
Now, use the geometric series formula on each factor of $\prod_{j=1}^{n-1} \big(\frac{\hat{z}_j}{1 - \hat{z}_j}\big)^{d_j-1}$ and retain terms with multidegree at most 1 in $y_1,\dots,y_{n-1}$ to show the equality of coefficients mentioned in the previous paragraph. Thus, we can conclude $\delta_1 = 0$.
\end{proof}

\begin{lemma}\label{lem:delta-2-boundary}
Let $x_1,\dots,x_n$ and $y_1,\dots,y_n$ be formal variables. Then $\delta_2$ equals the coefficient of $(x_1y_1)^{d_1 - 1} \cdots (x_ny_n)^{d_n - 1}$ in the formal power series
\begin{equation}\label{eq: expr delta2}
	Y\left[\left(e_2(\beta)+e_1(\gamma)-2e_1(\beta)\right)\prod_{i=1}^n\alpha_i+\sum_{i=1}^nY_i\alpha_i(2\beta_i-\gamma_i)\right]\,.
\end{equation}
\end{lemma}
\begin{proof}
For vectors $\bc,\bd$ each with $n$ nonnegative integer entries, let $\delta_r(\bc, \bd)$ be defined by
\[
\resizebox{\textwidth}{!}{
$\begin{aligned}
    &\sum_{\ba,\bb \ge 0} \left(\prod_{i=1}^n \binom{c_i}{a_i} \binom{d_i}{b_i+ 1} (-1)^{c_i - a_i + d_i - b_i - 1}\right)\left(\binom{e_1(\ba + \bb) - 2}{r} + (-1)^{r - 1}(r + 1)\right) \Vol_\sbZ(\boxed{a_1} \cdots \boxed{a_n}, b_1 \cdots b_n) \\
    & -\sum_{j=1}^n \max\left(0, d_j - \sum_{i \ne j} c_i\right)(-1)^{d_j - 1 + \sum_{i \ne j} c_i} \sum_{\substack{a_j \ge 0 \\ b_\ell\ge 0\\ \ell\neq j}} \binom{c_j}{a_j}(-1)^{c_j - a_j}\left(\prod_{i \ne j} \binom{d_i}{b_i + 1}(-1)^{d_i - b_i  - 1}\right)\cdot \\
    &\cdot\left(\binom{a_j + \sum_{i \ne j} b_i - 2}{r} + (-1)^{r - 1}(r + 1)\right)\Vol_\sbZ(\boxed{0} \cdots \setlength{\arraycolsep}{2pt}\begin{array}{|c|c|c|}
\hline
0 & a_j & 0 \\
\hline
\end{array} \cdots \boxed{0}, b_1 \cdots b_{j-1}\text{ } 0 \text{ }b_{j+1} \cdots b_n)
\end{aligned}$}
\]
Define
\[
	\eta \coloneqq \sum_{r \ge 0} \sum_{\substack{\bc \ge 0 \\ \bd \ge 0}} \delta_{r+1}(\bc,\bd)\,\bx^\bc\by^\bd t^r\,.
\]
By Lemmas \ref{lem:prismatic-eu} and \ref{lem:nonprismatic-eu} for the Euler obstructions, and Lemma \ref{lem:vol-gf} for the face volumes, we know that $\delta_{2}(\bd-\bone,\bd)$ equals $\delta_2$.
First, for every $i\in[n]$ we define the functions
\[
    f_i(t) \coloneqq \alpha_i\frac{\beta_i+(\beta_i^2-\gamma_i)t}{\beta_i-\gamma_it}\,,\quad f\coloneqq\prod_{i=1}^n f_i\,.
\]
An argument similar to the proof of Theorem \ref{lem:delta-1-strictly-balanced} allows us to write
\begin{equation}\label{eq: expr F}
	\eta = \frac{Y}{t(1 + t)^2}\left\{f - \prod_{i=1}^n\alpha_i - \sum_{i=1}^n Y_i(f_i - \alpha_i)\right\}\,.
\end{equation}
Now expand in $t$ to obtain an $O(t^2)$-approximation.
First, one obtains the $O(t^2)$-approximation
\[
	f_i(t) =  \alpha_i\left(1 + \beta_i\,t + \gamma_i\,t^2 + O(t^3)\right)\,.
\]
This yields the $O(t^2)$-approximation of their product
\[
    f = \prod_{i=1}^n \left[\alpha_i\left(1 + \beta_i\,t + \gamma_i\,t^2 + O(t^3)\right)\right] = \prod_{i=1}^n\alpha_i\cdot\left[1+e_1(\beta)\,t+\left(e_2(\beta)+e_1(\gamma)\right)t^2+O(t^3)\right]\,.
\]
Substituting this into \eqref{eq: expr F} we get, up to $O(t^2)$,
\begin{align*}
    \eta &= \frac{Y(1-2t)}{t}\left\{\left[e_1(\beta)\,t + \left(e_2(\beta)+e_1(\gamma)\right) t^2 \right]\prod_{i=1}^n\alpha_i-\sum_{i=1}^n Y_i\alpha_i(\beta_i\,t + \gamma_i\,t^2)\right\}\\
    &= Y(1-2t)\left\{\left[e_1(\beta) + \left(e_2(\beta) + e_1(\gamma)\right)t \right]\prod_{i=1}^n\alpha_i-\sum_{i=1}^n Y_i\alpha_i(\beta_i + \gamma_i\,t)\right\}\,.
\end{align*}
The coefficient of $t$ in the above expression coincides with \eqref{eq: expr delta2}.
\end{proof}

Recall the general notation introduced at the beginning of Section~\ref{subsec: degree A-discriminant strictly balanced}. Throughout the following proof, we also use
\begin{enumerate}
    \item $c \coloneqq \frac{1}{2}C = \sum_{i=1}^{n-1}(d_i - 1)$
    \item $B(z) \coloneqq \prod_{i=1}^n (1 + \hat{z}_i)^{d_i - 1} $
    \item $F_c(z) \coloneqq$ the $x^c$ coefficient in $\frac{x(2 - x - 3(1 + x)z)}{(1 - z - zx)^3}$
    \item $L(z) \coloneqq \sum_{i = 1}^n (C - 2d_j - 3)z_j$
    \item $R(z) = \frac{1}{2}\big(L(z)^2+ \sum_{i=1}^n (d_i - 1) \hat{z}_i^2 - \sum_{j=1}^n d_j z_j^2\big)$
    \item $U_W \coloneqq \sum_{i=1}^{n-1} z_i$
    \item $W \coloneqq \sum_{i=1}^{n-1} (d_i - 1)z_i$
    \item $X \coloneqq \sum_{i=1}^{n-1} d_iz_i$
    \item $\Sigma_{12} \coloneqq$ coefficient of $\prod_{i=1}^n(x_iy_i)^{d_i - 1}$ in $Y(e_1(\gamma)-2\,e_1(\beta) + e_2(\beta))\prod_{i=1}^n\alpha_i$
    \item $\Sigma_3 \coloneqq$ coefficient of $\prod_{i=1}^n(x_iy_i)^{d_i - 1}$ in $Y\sum_{i=1}^n Y_i\alpha_i(2\beta_i-\gamma_i)$
\end{enumerate}

\begin{theorem}\label{thm:delta-2-boundary}
Let $\bd-\bone\coloneqq (d_1 - 1,\dots,d_{n-1} - 1)$.
Let $\binom{d_n - 1}{d_1 - 1,\dots,d_{n-1} - 1}=\frac{(d_n - 1)!}{(d_1 - 1)! \cdots (d_{n-1} - 1)!}$. We have
\[
    \delta_2 = \frac{1}{2} \binom{d_n - 1}{d_1 - 1,\dots,d_{n-1} - 1}\left(d_n(d_n - 1) + (4\,d_n - 1)e_2(\bd-\bone) + 2\,e_2^2(\bd-\bone)\right)\,
\]
where $e_2(\bd-\bone) = \sum_{1 \le i < j \le n - 1}(d_i - 1)(d_j - 1)$.
\end{theorem}
\begin{proof}
As in the corresponding proof for the strictly balanced case (Theorem~\ref{thm:delta-1-strictly-balanced}), we change variables. Applying Lemma \ref{lem:claim-3} and the geometric series formula (specifically used to show the $x_i^{d_i - 1}$ coefficient in $\frac{1 - \hat{z}_i}{1 - (1 + x_i)\hat{z}_i}$ is $\frac{\hat{z}_i^{d_i - 1}}{1 - \hat{z}_i}$ and the $x_i^{d_i - 1}$ coefficient in $\frac{x_i}{(1 - (1 + x_i)\hat{z}_i)^2}$ is $\frac{(d_i - 1) \hat{z}_i^{d_i - 2}}{(1 - \hat{z}_i)^{d_i}}$, and $F_c(z) = \binom{c + 1}{2} \frac{(2 - 3z)z^{c - 1}}{(1 - z)^{c + 2}} - \binom{c}{2} \frac{(1 + 3z)z^{c - 2}}{(1 - z)^{c + 1}}$), and letting $B_i(z)$ denote the degree $i$ part of $B(z)$, we may write that $\Sigma_{12}$ is the $\bz^\bc$ coefficient of
\[
    B_{C - 2} + \left((C - 3)L + \sum_i (d_i - 1) \hat{z}_i\right)B_{C - 1} + \frac{C(C - 5)}{2}B_C
\]
or equivalently of $\bigg(\frac{C(C - 5)}{2} + (C - 3)L(z) + \sum_{i=1}^n (d_i - 1) \hat{z}_i + R(z)\bigg) B(z)$.

Now we try to eliminate the variable $d_n - 1$, since it is the sum of the other $(d_i - 1)$'s. Consider $B$ and write it as $(1 + \hat{z}_n)^{d_n - 1} \prod_{i=1}^{n-1}(1 + \hat{z}_i)^{d_i - 1}$. Using the binomial theorem to calculate the coefficient of the $z_n^c$ term as $1$ and the coefficient of the $z_n^{c-1}$ term as $c + c\,U_W - W$, we find that $\Sigma_{12} $ is the $z_1^{d_1 - 1} \cdots z_{n-1}^{d_{n-1} - 1}$ coefficient of
\[
    (1 + U_W)^c[c(c - 2) + (2\,c^2 - 2\,c + 3)\,U_W - 2(c - 1)\,W + (c\,U_W - W)^2 + e_2(\bd-\bone)]
\]
or equivalently
\[
    \Sigma_{12} = \frac{1}{2}\binom{d_n - 1}{d_1 - 1,\dots,d_{n-1} - 1}(2\,c^2 + 2\,c + 4\,e_2(\bd-\bone)^2 + 7(c - 1)\,e_2(\bd-\bone) - 3\,e_3(\bd-\bone))\,.
\]
Now we consider $\Sigma_3$. Since the $x_j^{d_j - 1}$ coefficient of $\frac{1}{1 - x_jy_i}$ is $y_i^{d_j - 1}$, the $\bx$ coefficient of the $i$-th summand of $\Sigma_3$ is $y_i^{C - (d_i - 1)}F_{d_i - 1}(\hat{z}_i)$. Since we need $d_i - 1 = c$, we must have $i = n$. The other contribution to $\by^{\bd-\bone}$ comes from $\prod_{i=1}^{n-1}(1 - z_i)^{d_i}F_c(U_W)$ and $F_c(U_W)$, whose terms of degrees $(c - 2)$, $(c - 1)$, and $c$ are respectively $-\binom{c}{2}$, $-\frac{1}{2}c(c + 3)(c - 2)$, and $-\frac{1}{4}c(c + 1)(c + 5)(c - 2)$. Similarly, a low-order approximation (specifically, a quadratic approximation) of $\prod_{i=1}^{n-1}(1 - z_i)^{d_i}$ can be obtained; after some standard combinatorial computations, we obtain
\[
    \Sigma_3 = -\frac{1}{2}\binom{d_n - 1}{d_1 - 1,\dots,d_{n-1} - 1}(2\,e_2(\bd-\bone)^2 + (3\,c - 10)\,e_2(\bd-\bone) - 3\,e_3(\bd-\bone) + c^2 + c)\,.
\]
Finally,
\[
    \Sigma_{12} + \Sigma_3 = \frac{1}{2}\binom{d_n - 1}{d_1 - 1,\dots,d_{n-1} - 1}(c^2 + c + (4\,c + 3)\,e_2(\bd-\bone) + 2\,e_2(\bd-\bone)^2)\,,
\]
which exactly matches the $\delta_2$ of the theorem statement.
\end{proof}

\begin{example}\label{ex: 223}
When $(d_1,d_2,d_3) = (2,2,3)$,
\[
    \delta_2 = \frac{1}{2} \binom{2}{1,1} (3 \cdot 2 + (4 \cdot 3 - 1)(1 \cdot 1) + 2(1 \cdot 1)^2) = 19\,.
\]
\end{example}

\begin{corollary}\label{cor:delta-2-nonzero}
Let $e_2(\bd-\bone) = \sum_{1 \le i < j \le n - 1}(d_i - 1)(d_j - 1)$. Then $\codim(\nabla_{\!\ka})=2$ and
\begin{equation}\label{eq:deg-D2}
    \deg(\nabla_{\!\ka}) = \frac{1}{2} \binom{d_n - 1}{d_1 - 1,\dots,d_{n-1} - 1}\left(d_n(d_n - 1) + (4d_n - 1)e_2(\bd-\bone) + 2 e_2(\bd-\bone)^2\right)\,.
\end{equation}
\end{corollary}
\begin{proof}
The summands in Theorem \ref{thm:delta-2-boundary} are manifestly nonnegative; now apply \cite[Theorem 1.4]{matsui2011geometric}.
\end{proof}

\begin{example}\label{ex: 223 part 2}
Consider the format $\bd=(2,2,3)$ again. The Nashmatrix $\ka$ has format $r\times s=11\times 20$. As a consequence of Example~\ref{ex: 223} and Corollary \ref{cor:delta-2-nonzero}, the $\ka$-discriminant variety is 2-codimensional of degree 19. Its vanishing ideal involves 20 unknowns, namely the coefficients (with $i,j\in[2]$ and $k\in[3]$)
\[
    a_{jk}^{(1)} \coloneqq X_{2jk}^{(1)}-X_{1jk}^{(1)}\,\quad a_{ik}^{(2)} \coloneqq X_{i2k}^{(2)}-X_{i1k}^{(2)}\,\quad a_{ij}^{(3)} \coloneqq X_{ij2}^{(3)}-X_{ij1}^{(3)}\,\quad a_{ij}^{(4)} \coloneqq X_{ij3}^{(3)}-X_{ij1}^{(3)}
\]
of the polynomials defining the multihomogeneous ideal $\langle\Delta f_{1,2}^{(1)},\Delta f_{1,2}^{(2)},\Delta f_{1,2}^{(3)},\Delta f_{1,3}^{(3)}\rangle$. The discriminant of the last two equations is the classical $2\times 2\times 2$ hyperdeterminant
\begin{equation}\label{eq: hyperdet 222}
    \left(a_{11}^{(3)}a_{22}^{(4)}+a_{22}^{(3)}a_{11}^{(4)}-a_{12}^{(3)}a_{21}^{(4)}-a_{21}^{(3)}a_{12}^{(4)}\right)^2
    -4\left(a_{11}^{(3)}a_{22}^{(3)}-a_{12}^{(3)}a_{21}^{(3)}\right)\left(a_{11}^{(4)}a_{22}^{(4)}-a_{12}^{(4)}a_{21}^{(4)}\right)\,,
\end{equation}
and its zero locus forms the unique hypersurface component of the Nash discriminant $\Delta(2,2,3)$; see \cite[Example 3.27]{abo2026vector}.
Using the supplementary code on \cite{Zenodo}, we sampled the point $[a_{11}^{(1)}:\cdots:a_{22}^{(4)}]\in\bP^{19}$ with the 20 coordinates displayed in Table~\ref{table: coefficients format 2x2x3 on A-discriminant}.
This point can be realized by the game $X=(X^{(1)},X^{(2)},X^{(3)})$ with the payoff tensors displayed in Table~\ref{table: payoff tensors 2x2x3 game on A-discriminant}. One verifies that the $2\times 2\times 2$ hyperdeterminant in \eqref{eq: hyperdet 222} does not vanish at this point, namely, the point is not contained in the one-codimensional part of the Nash discriminant variety $\Delta(2,2,3)$.
We verified that the Nash equilibrium scheme has two components:
\[
    \kz_X = \left(\{[5:18]\}\times\{[1:1]\}\times\{[6:28:5]\}\right) \cup \left(\{[1:2]\}\times\{[1:2]\}\times L\right)\subseteq\bP^1\times\bP^1\times\bP^2\,,
\]
namely a point and a line not containing the point, where $L\subseteq\bP^2$ is the line defined by $2\,\pi_1^{(3)}-\pi_2^{(3)}+4\,\pi_3^{(3)}=0$. The first point corresponds to the totally mixed Nash equilibrium
\[
    \left(\left(\frac{5}{23},\frac{18}{23}\right),\left(\frac{1}{2},\frac{1}{2}\right),\left(\frac{2}{13},\frac{28}{39},\frac{5}{39}\right)\right)\in \Delta_1^\circ\times\Delta_1^\circ\times\Delta_2^\circ\,.
\]
Furthermore, one verifies that all points in the open line segment
\[
    \left\{\left(\left(\frac{1}{3},\frac{2}{3}\right),\left(\frac{1}{3},\frac{2}{3}\right),\left(\frac{1-5t}{3},\frac{2(t+1)}{3},t\right)\right)\ \bigg|\ t\in\left(0\,,\frac{1}{5}\right)\right\}\subseteq \Delta_1^\circ\times\Delta_1^\circ\times\Delta_2^\circ
\]
are totally mixed Nash equilibria of $X$. The code used for this example is \texttt{code\_example\_3.51.m2}, available on \cite{Zenodo}.
\end{example}

\begin{table}[htbp]
\centering
\begin{tabular}{crr}
\toprule
 & $a_{jk}^{(1)}$ & $a_{ik}^{(2)}$ \\
\midrule
11 & $ 24$ & $-24$ \\
12 & $ -4$ & $  4$ \\
13 & $  0$ & $ -8$ \\
21 & $-20$ & $  4$ \\
22 & $  6$ & $  2$ \\
23 & $-16$ & $-12$ \\
\bottomrule
\end{tabular}
\hspace{40pt}
\begin{tabular}{crr}
\toprule
 & $a_{ij}^{(3)}$ & $a_{ij}^{(4)}$ \\
\midrule
11 & $-4$ & $20$ \\
12 & $ 4$ & $-2$ \\
21 & $ 2$ & $-2$ \\
22 & $-2$ & $-3$ \\
\bottomrule
\end{tabular}
\caption{A point on the $\ka$-discriminant variety of format $\bd=(2,2,3)$.}
\label{table: coefficients format 2x2x3 on A-discriminant}
\end{table}

\begin{table}[htbp]
\centering
\begin{tabular}{crrr}
\toprule
$(j_1 j_2 j_3)$ & $X^{(1)}$ & $X^{(2)}$ & $X^{(3)}$ \\
\midrule
111 & $  0$ & $  1$ & $ 1$ \\
112 & $  1$ & $  0$ & $-3$\\
113 & $ -1$ & $ -1$ & $21$ \\
121 & $ -1$ & $-23$ & $ 1$ \\
122 & $ -1$ & $  4$ & $ 5$ \\
123 & $  1$ & $ -9$ & $-1$ \\
211 & $ 24$ & $  1$ & $-1$ \\
212 & $ -3$ & $ -1$ & $ 1$ \\
213 & $ -1$ & $  1$ & $-3$ \\
221 & $-21$ & $  5$ & $ 1$ \\
222 & $  5$ & $  1$ & $-1$ \\
223 & $-15$ & $-11$ & $-2$ \\
\bottomrule
\end{tabular}
\caption{Payoff tensors of a three-player game of format $\bd=(2,2,3)$ corresponding to a point on the $\ka$-discriminant.}
\label{table: payoff tensors 2x2x3 game on A-discriminant}
\end{table}

\section{Mixed Segre embeddings are polynodal}\label{sec: polynodality}

Recall the notations in Section \ref{sec: prelim}. Let $X=(X^{(1)},\ldots,X^{(n)})$ be an $n$-player game of format $\bd$. Note that each generator $\Delta f_{1,k}^{(i)}$ of the ideal defining the Nash equilibrium scheme $\kz_X$ is linear in every vector of variables $\pi^{(j)}$ for $j\neq i$, while $\pi^{(i)}$ does not appear. As in \cite[Section 5]{pratt2026multigraded}, in this section we describe the equations of $\kz_X$ as linear equations in a particular embedding of $\Pd$, and we interpret $\nabla_{\!\ka}$ as a {\em multigraded Hurwitz locus}, see \cite[Section 2]{pratt2026multigraded}. Our goal is to revisit the degree formula of Corollary \ref{cor:discriminant-strictly-balanced}. More precisely, we show in Corollary \ref{cor: A-discriminant and multigraded Hurwitz form} that this degree equals the sum of the components on the right-hand side of \cite[Eq. (26)]{pratt2026multigraded}.

The first ingredient is the \emph{mixed Segre embedding}
\begin{equation}\label{eq: mixed Segre embedding}
    \begin{array}{rccc}
    \sigma_\bd \colon & \Pd=\prod_{i=1}^n\bP^{d_i-1} & \lhook\joinrel\xrightarrow{\hspace{.35cm}} & \bP\coloneqq\prod_{i=1}^n \bP\left(\bigotimes_{j\neq i}\bC^{d_j}\right)\\
	& ([x_1],\ldots,[x_n]) & \longmapsto & ([\otimes_{j\neq i} x_j])_{i=1}^n
    \end{array}\,.
\end{equation}
In particular $\sigma_\bd=(\sigma_{\bd,1},\dots,\sigma_{\bd,n})$, where each $\sigma_i$ is the classical Segre embedding
\[
    \sigma_{\bd,i}\colon \bP_{-i}^\bd\coloneqq\prod_{j\neq i}\bP^{d_j-1} \lhook\joinrel\xrightarrow{\hspace{.35cm}} \bP^{r_i}=\bP\left(\bigotimes_{j\neq i}\bC^{d_j}\right)\,.
\]
Let $X\subseteq\bP$ be the image of $\sigma$. In particular, $X$ is nonsingular of dimension $d_1+\cdots+d_n-n$.
For any $\beta\coloneqq(\beta_1,\ldots,\beta_n)\in\bN^n$, we consider $\bG_\beta\coloneqq\prod_{i=1}^n\bG(\beta_i,r_i)$, where $\bG(\beta_i,r_i)$ denotes the Grassmannian of $\beta_i$-dimensional linear subspaces of an $r_i$-dimensional projective space. Fix $\alpha\coloneqq(r_1-d_1+1,\ldots,r_n-d_n+1)\in\bN^n$. For each $i\in[n]$, consider $L=L_1\times\cdots\times L_n\in\bG_{\alpha+e_i}$, where $e_i$ is the $i$-th unit vector in $\bN^n$, and define $C_i = C_i(L) \coloneqq X \cap L.$
For a generic $L\in\bG_{\alpha+e_i}$, $C_i$ is the complete intersection between $X$ and the product of linear spaces $L_1\times\cdots\times L_n$ of codimension $d_1+\cdots+d_n-n-1$ in $\bP$. Since the $L_i$'s are generic, the intersection $C_i = X \cap L$ is a nonsingular curve in $\bP$.

\medskip

In the proof of Theorem \ref{thm: polynodal general}, we utilize the following Bertini-type theorem on degeneracy loci of vector bundle morphisms (see \cite[Theorem 2.8]{ottaviani1995varieta}).

\begin{theorem}\label{thm: Bertini-type}
Given vector bundles $\ke,\kf$ on a variety $X$, and assuming that $\ke^\vee\otimes\kf$ is globally generated, then for a generic morphism $\varphi\colon\ke\to\kf$ the degeneracy locus
\[
    D_k(\varphi)\coloneqq\{p\in X\mid \rank\varphi_p\le k\}
\]
is either empty or has the expected dimension
\[
    \dim(D_k(\varphi)) = \dim(X)-(\rank\ke-k)(\rank\kf-k)\,.
\]
Furthermore $\Sing D_k(\varphi)\subseteq D_{k-1}(\varphi)$. In particular, if $\dim(X)<(\rank\ke-k+1)(\rank\kf-k+1)$, then $D_k(\varphi)$ is nonsingular for a generic $\varphi$. 
\end{theorem}

The following result shows that any projection of $C_i$ on the $j$-th factor of $\bP$ is nonsingular. This implies immediately that $X$ is \emph{$\alpha$-polynodal}, meaning that for any $i\in[n]$ and for a generic $L\in\bG_{\alpha+e_i}$, any projection of $C_i$ has at most nodal singularities (see \cite[Section 3]{pratt2026multigraded}).

\begin{theorem}\label{thm: polynodal general}
Let $n\ge 3$ and let $\bd$ be a strictly balanced game format. For a generic $L\in\bG_{\alpha+e_i}$ and for every $j\in[n]$, the projection of $C_i$ to the $j$-th component of $\bP$ is a nonsingular curve.
\end{theorem}
\begin{proof}
Recall that $r_i=-1+\prod_{j\neq i}d_j$ for all $i\in[n]$ and $\alpha=(r_1-d_1+1,\ldots,r_n-d_n+1)$. We study only the case $i=1$ because the other cases are similar. Consider the curve $C_1=X\cap L$, where $L=L_1\times\cdots\times L_n\subseteq\bP$ with $\dim(L_1)=\alpha_1+1=r_1-d_1+2$ and $\dim(L_j)=\alpha_j=r_j-d_j+1$ for all $j\ge 2$.
The mixed Segre embedding $\sigma_\bd$ is an isomorphism when restricted to its image $\sigma_\bd(\Pd)=X$. In particular, instead of $C_1$ we consider its pullback $C'_1\subseteq\Pd$. Define $\bone_j\in\bN^n$ as the vector with all components $1$ except the $j$-th, which is $0$. Then $C'_1$ equals the zero locus of the general section $s\in H^0(\Pd,\ke_1)$, where 
\begin{equation}
	\ke_1\coloneqq \bigoplus_{j=1}^n\ko_{\Pd}(\bone_j)^{\oplus(d_j-1-\delta_{1j})}
\end{equation}
and $\delta_{1j}$ is the Kronecker delta.
In particular we write $s=(s_1,\ldots,s_n)$, where
\begin{equation}\label{eq: sections s_j defining C'_1 2}
    s_j=\left(s_{(j,1)},\ldots,s_{(j,d_j-1-\delta_{1j})}\right)\quad\forall\,j\in[n]\,.
\end{equation}
Since $s$ is generic, the curve $C'_1$ is nonsingular. The Chow ring of $\Pd$ is
\[
    A^\ast(\Pd)\cong\frac{\bZ[h_1,\ldots,h_n]}{\langle h_1^{d_1},\ldots,h_n^{d_n}\rangle}\,,
\]
where $\int_{\Pd}h_1^{d_1-1}\cdots h_n^{d_n-1}=1.$ For each $j\in[n]$, let $\hat{h}_j\coloneqq c_1(\ko_{\Pd}(\bone_j))=\sum_{\ell\neq j} h_\ell.$ Then
\begin{equation}\label{eq: class of C'_1}
    [C'_1] = \prod_{j=1}^n\hat{h}_j^{d_j-1-\delta_{1j}}\in A^\ast(\Pd)\,.
\end{equation}
Recall that $\bP_{-j}^\bd$ is the product of all but the $j$-th factor of $\Pd$ and $\pr_{-j}\colon\Pd\to\bP_{-j}^\bd$ is the corresponding projection. Our goal is to show that $\pr_{-j}(C'_1)$ is a nonsingular curve. 

\medskip

For all $\ell\neq j$ and $t\in[d_\ell-1-\delta_{1\ell}]$, write each component $s_{(\ell,t)}$ of $s_\ell$ in \eqref{eq: sections s_j defining C'_1 2} as a linear form $s_{(\ell,t)} = \sum_{p=0}^{d_j-1}a_p^{(\ell,t)}\,x_{jp}$ in $x_j=(x_{j,0},\ldots,x_{j,d_j-1})$, where the coefficients $a_p^{(\ell,t)}$ are global sections of $\ko_{\bP_{-j}^\bd}(\bone_\ell)$ and $\bone_\ell$ is defined similarly as before, with the difference that now $\bone_\ell\in\bN^{n-1}$. In other words, the coefficients $a_p^{(\ell,t)}$ are multilinear forms in $(x_1,\ldots,x_n)$ which exclude $x_j$ and $x_\ell$. The global sections $s_\ell$ written in this form define a vector bundle morphism on $\bP_{-j}^\bd$
\[
	\varphi_1\colon\ko_{\bP_{-j}^\bd}^{\oplus d_j} \longrightarrow \bigoplus_{\ell\neq j}\ko_{\bP_{-j}^\bd}(\bone_\ell)^{\oplus(d_\ell-1-\delta_{1\ell})}\,.
\]
For all $\ell\neq j$, define the $d_j\times(d_\ell-1-\delta_{1\ell})$ matrix of polynomials
\[
    A_\ell \coloneqq
    \begin{pmatrix}
    a_0^{(\ell,1)} & \cdots & a_0^{(\ell,d_\ell-1-\delta_{1\ell})}\\
    \vdots & & \vdots\\
    a_{d_j-1}^{(\ell,1)} & \cdots & a_{d_j-1}^{(\ell,d_\ell-1-\delta_{1\ell})}
    \end{pmatrix}\,.
\]
Then for any $[x_{-j}]\coloneqq([x_1],\ldots,[x_{j-1}],[x_{j+1}],\ldots,[x_n])\in\bP_{-j}^\bd$ the fiber of $\varphi_1$ at $[x_{-j}]$ is the linear map defined by the transpose of the $d_j\times\sum_{\ell\neq j}(d_\ell-1-\delta_{1\ell})$ matrix of polynomials
\[
    \left(
	\begin{array}{c|c|c|c|c|c}
    A_1 & \cdots & A_{j-1} & A_{j+1} & \cdots & A_n
    \end{array}
    \right)
\]
obtained by concatenating the blocks $A_\ell$ horizontally.
In particular $\pr_{-j}(C'_1)$ is the intersection between the zero locus of the section $s_j$, which we can see as a section of $\ko_{\bP_{-j}^\bd}(\bone_j)^{\oplus(d_j-1-\delta_{1j})}$, and the degeneracy locus
\[
	D_{d_j-1}(\varphi_1) = \{[x_{-j}]\in\bP_{-j}^\bd\mid \rank\varphi_1(x_{-j})\le d_j-1\}\,.
\]
We can apply Theorem~\ref{thm: Bertini-type} with $\ke=\ko_{\bP_{-j}^\bd}^{\oplus d_j}$, $\kf=\bigoplus_{\ell\neq j}\ko_{\bP_{-j}^\bd}(\bone_\ell)^{\oplus(d_\ell-1-\delta_{1\ell})}$ and $\varphi_1\colon\ke\to\kf$ because $\ke^\vee\otimes\kf$ is globally generated. First, a consequence of \cite[Theorem 4.1]{mckelvey1997maximal} (see also \cite[Lemma 3.5]{abo2026vector}) is that $D_{d_j-1}(\varphi_1)$ is nonempty for a generic $\varphi_1$. Therefore, a generic $\varphi_1$ is such that
\begin{align*}
    \dim(D_{d_j-1}(\varphi_1)) &= \dim(\bP_{-j}^\bd)-(\rank\ke-(d_j-1))(\rank\kf-(d_j-1))\\
    &= \sum_{\ell\neq j}(d_\ell-1)-\left(\sum_{\ell\neq j}(d_\ell-1-\delta_{1\ell})-(d_j-1)\right) = d_j-1+\sum_{\ell\neq j}\delta_{1\ell} = d_j-\delta_{1j}\,.
\end{align*}
Furthermore, we have that $\Sing D_{d_j-1}(\varphi_1)\subseteq D_{d_j-2}(\varphi_1)$, where
\begin{align*}
    \dim(D_{d_j-2}(\varphi_1)) &= \dim(\bP_{-j}^\bd)-(\rank\ke-(d_j-2))(\rank\kf-(d_j-2))\\
    &= \sum_{\ell\neq j}(d_\ell-1)-2\left(\sum_{\ell\neq j}(d_\ell-1-\delta_{1\ell})-(d_j-2)\right)\\
    &= 2(d_j-1-\delta_{1j})-\sum_{\ell\neq j}(d_\ell-1)\,.
\end{align*}
Consider also the zero locus of $s_j$, which is a generic complete intersection of codimension $d_j-1-\delta_{1j}$ in $\bP_{-j}^\bd$.
On the one hand, by Bertini's theorem, the intersection of the zero locus of $s_j$ with $D_{d_j-1} \setminus D_{d_j-2}$ is nonsingular. On the other hand, the intersection between $D_{d_j-2}(\varphi_1)$ and the zero locus of $s_j$ is empty since
\[
    \dim(D_{d_j-2}(\varphi_1))-(d_j-1-\delta_{1j}) = d_j-1-\delta_{1j}-\sum_{\ell\neq j}(d_\ell-1) < 0
\]
whenever $d_j-1<\sum_{\ell\neq j}(d_\ell-1)$, which is satisfied by our assuptions. Hence, the intersection of the zero locus of $s_j$ with $D_{d_j-1} \setminus D_{d_j-2}$ equals the intersection of the zero locus of $s_j$ with $D_{d_j-1}$. Equivalently $\pr_{-j}(C'_1)$ is a nonsingular curve. 
This also implies that the projection of $C_1$ onto the $j$-th component of $\bP$ is a nonsingular curve.
\end{proof}

The following corollary follows immediately from Theorem \ref{thm: polynodal general} after applying \cite[Theorem 3.4 and Proposition 5.5]{pratt2026multigraded}. Observe that in \cite{pratt2026multigraded} the integers $d_i$ are denoted by $k_i+1$, while our $n$ corresponds to $\ell$.

\begin{corollary}\label{cor: A-discriminant and multigraded Hurwitz form}
Consider the Nashmatrix $\ka$ of a strictly balanced game format $\bd$. The degree of $\nabla_{\!\ka}$ of Corollary \ref{cor:discriminant-strictly-balanced} equals $u_1+\cdots+u_n$, where $u_i$ is the coefficient of $\prod_{i=1}^n z_i^{d_i-1}$ in
\[
    2\prod_{j=1}^n\hat{z}_j^{d_j-1}+\frac{1}{\hat{z}_i}\prod_{j=1}^n\hat{z}_j^{d_j-1}\left(-\hat{z}_i+\sum_{j=1}^n \left[(d_j-1)\hat{z}_j-d_jz_j\right]\right)\,.
\]
\end{corollary}

For example, when $d_1=\cdots=d_n=2$, we revisit the formula of Lemma \ref{lem:binary-derangement}: the components of the multidegree of $\nabla_{\!\ka}$, seen as a multigraded Hurwitz locus, are
\[
    u_1 = \cdots = u_n = [(n - 2)\,\cdot\,!n + (n - 3)\,\cdot\,!(n - 1)]\,.
\]

\begin{remark}
This remark concerns the genus of $C'_i$ and its projections $\pr_{-j}(C'_i)$, which we discuss in the binary case $d_1=\cdots=d_n=2$ for simplicity.
Consider the class of $C'_i$, which is similar to \eqref{eq: class of C'_1}.
Since $K_{\Pd}=-2\sum_{i=1}^nh_i$, adjunction gives
\[
	K_{C'_i} = \left.\left(K_{\Pd}+\sum_{j\neq i}\hat{h}_j\right)\right|_{C'_i} = \left.\left((n-3)h_i+(n-4)\hat{h}_i\right)\right|_{C'_i}\,.
\]
Hence
\[
	2\,g(C'_i)-2 = \int_{\Pd}\left((n-3)h_i+(n-4)\hat{h}_i\right)\prod_{j\neq i}\hat{h}_j\,.
\]
Let $!m$ be the number of derangements of $m$ elements. Then the coefficient computation gives
\[
	2\,g(C'_i)-2 = (n-3)\,\cdot\,!(n-1)+(n-4)\,\cdot\,!n\,.
\]
Equivalently,
\[
	g(C'_i) = 1+\frac{(n-3)\,\cdot\,!(n-1)+(n-4)\,\cdot\,!n}{2}\,.
\]
For example, $g(C'_i)=0,0,2,32,332$ for $n=2,3,4,5,6$ respectively.

As a consequence of Theorem~\ref{thm: polynodal general}, for a generic $L$, the morphisms $f_j'\colon C'_i \to \pr_{-j}(C'_i)\subseteq\bP_{-j}$ for all $j\in[n]$ are nonconstant, unless $n=2$ and $(i,j)\in\{(1,2),(2,1)\}$.
Since $C'_i$ is a curve, the nonconstant morphism $f_j'$ from $C'_i$ is generically finite. Finally, since $C'_i$ is projective, then $f_j'$ is proper (see \cite[Theorem II.4.9]{hartshorne1977algebraic}).
Since the fibers of $f'_j$ are points and $f'_j$ is proper, we conclude that $f_j'$ and hence $f_j$ are finite (see for example \cite[Lemma 30.21.1]{stacks-project}). Since $f_j$ and $f_j'$ are isomorphisms for generic $L$ (unless $n=2$ and $(i,j)\in\{(1,2),(2,1)\}$), we conclude that $g(f_j'(C'_i))=g(C'_i)$ and $g(f_j(C_i))=g(C_i)$.
\end{remark}

\section{The irreducible components of the Nash discriminant variety}\label{sec:aps-3.8}

For balanced game formats with at least three players, we answer \cite[Problem 3.8]{abo2026vector} by describing the irreducible components of the Nash discriminant variety $\Delta(\bd)$ (see \cite[Definition 3.2]{abo2026vector}) and their degrees. Our main tools are \cite[Corollary 3.6 and Theorem 1.5]{pokidkin}. 

\medskip

For convenience, we reproduce some definitions from \cite{pokidkin}. Given a tuple of finite sets $\ma$ of a common ambient lattice and a subtuple $\mb$, the \emph{defect} of $\mb$ is the number
\[
    \delta(\mb) \coloneqq \dim\left(\text{the affine span of the Minkowski sum $\sum_{i \in I} A_i$}\right) - |I|\,,
\]
where $I$ is the index set of $\mb$. A tuple of finite sets is \emph{irreducible} if the defects of all proper subtuples are positive.
A tuple of finite sets is \emph{independent} if the defects of all subtuples are nonnegative. 
Independent tuples of zero defect are called \emph{BK-tuples}. We direct readers to \cite{pokidkin} for remaining definitions, including the definition of the $\ma$-discriminant. 

From here on, we assume that all systems $\ma$ are square in the sense of \cite{pokidkin}.

\begin{corollary}[{\cite[adapted from Corollary 3.6]{pokidkin}}]\label{cor:pokidkin}
For a nonlinear irreducible BK-tuple $\ma$, the $\ma$-discriminant, the Cayley discriminant, and the mixed discriminant coincide.
\end{corollary}

\begin{theorem}[{\cite[Theorem 1.5]{pokidkin}}]\label{thm:pokidkin}
For a BK-tuple $\ma$, the $\ma$-discriminant has the following distinct components:
\begin{enumerate}
    \item the Cayley discriminants of the nonprelinear simple BK-subtuples, which have codimension one;
    \item the Cayley discriminants of the prelinear simple BK-subtuples that are not contained in nonprelinear simple BK-subtuples, which have codimension two.
    \end{enumerate}
\end{theorem}

\begin{remark}
    Terms such as ``simple'' and ``prelinear'' appear in Theorem \ref{thm:pokidkin} but their definitions do not appear in this paper; instead, we direct readers to \cite[Definition 1.4]{pokidkin}. We leave these terms undefined because it is faster to verify that the $\ma$-discriminant has two irreducible components using previous work than to directly identify which BK-subtuples of $\ma$ are nonprelinear simple and prelinear simple. 
\end{remark}

\subsection{Comparing \texorpdfstring{$D_\ma$}{ma-discriminant} with \texorpdfstring{$\nabla_{\!\ka}$}{nabla ka}}
In this section, we formally define the Nash $\ma$-discriminant $D_\ma$ and compare it with the $\ka$-discriminant $\nabla_{\!\ka}$ for Nashmatrix $\!\ka$. 
\begin{remark}
In general, the $\ka$-discriminant is irreducible, but the $\ma$-discriminant is not necessarily irreducible.
\end{remark}
In an $\ma$-discriminant, $\ma$ is a tuple of $\ell$ finite sets $(A_1, \dots, A_\ell)$, as given in \cite[Introduction]{pokidkin}. In the \emph{Nash $\ma$-discriminant}, these finite sets represent point configurations given by the vertices of Newton polytopes of the dehomogenized multihomogeneous equations
\[
    \left\{\frac{\Delta f_{1,k}^{(i)}}{\prod_{j \in [n] \setminus \{i\}} \pi_1^{(j)}}\right\}_{1 \le i \le n,\,2 \le k \le d_i}
\]
which are products of simplices of the form
\[
    \Delta^{(i)} \coloneqq \Delta_{d_1 - 1} \times \cdots \times \Delta_{d_{i-1} - 1} \times \{0\} \times \Delta_{d_{i+1} - 1} \times \cdots \times \Delta_{d_n - 1}\,,
\]
each occurring $d_i - 1$ times for every $i \in [n]$.
Consequently, we obtain a $(d_1 + \cdots + d_n - n)$-tuple
\begin{equation}\label{eq: Newton polytopes for multihomog}
    \Delta[d_1, \ldots, d_n] \coloneqq (\underbrace{\Delta^{(1)}, \ldots, \Delta^{(1)}}_{d_1 - 1\ \text{times}},\underbrace{\Delta^{(2)}, \ldots, \Delta^{(2)}}_{d_2 - 1\ \text{times}},\ldots,\underbrace{\Delta^{(n)}, \ldots, \Delta^{(n)}}_{d_n - 1\ \text{times}}).
\end{equation}

Here it turns out that $\ell = \sum_{i=1}^n(d_i-1)$. Each simplex $\Delta^{(i)}$ lies in a $(d_i - 1)$-dimensional space, and each product of simplices lies in a $\sum_{i = 1}^n (d_i -1)$-dimensional space.

\begin{example}
For a game of format $\bd = (2,2,3)$,
\[
\resizebox{\textwidth}{!}{
    $\begin{aligned}
    A_1 = \conv
    \begin{pmatrix}
    0&0&0&0&0&0 \\
    0&0&0&1&1&1 \\
    0&1&0&0&1&0 \\
    0&0&1&0&0&1
    \end{pmatrix},\,
    A_2 = \conv
    \begin{pmatrix}
    0&0&0&1&1&1 \\
    0&0&0&0&0&0 \\
    0&1&0&0&1&0 \\
    0&0&1&0&0&1 
    \end{pmatrix},\,
    A_3 = A_4 = \conv
    \begin{pmatrix}
    0&0&1&1 \\
    0&1&0&1 \\
    0&0&0&0 \\
    0&0&0&0 
    \end{pmatrix}
    \end{aligned}\,.$}
\]
\end{example}

\begin{remark}
    One may compare this construction with that of the Nashmatrix. The equations $\Delta f^{(i)}_{1,k}$ in Definition \ref{def: Nashmatrix} are multihomogeneous, but here we dehomogenize. 
    The point of dehomogenizing is to produce a square system so that we can apply the results of \cite{pokidkin}.
    In particular, it is clear that (de)homogenization fixes the count of irreducible components as well as their degrees. 
\end{remark}

Let $\Delta(E)$ be the discriminant of the vector bundle $E$ over $\Pd$ defined in \cite[Eq. (2.4)]{abo2026vector}. As explained in \cite[Remark 3.4]{abo2026vector}, $\Delta(\bd)$ is a projective cone over $\Delta(E)$, so their invariants, such as codimension and degree of irreducible components, agree.

\begin{lemma}
     The Nash $\ma$-discriminant $D_\ma$ equals the affine cone over $\Delta(E)$, defined as
     \[
        \cone(\Delta(E)) \coloneqq \{f \in H^0(\Pd, E) \mid \text{$f = 0$  or $[f] \in \Delta(E)$}\}\,.
     \]
\end{lemma}
\begin{proof}
    To see that $D_\ma \subseteq \cone(\Delta(E))$, consider the set $D_\ma^\circ$ of global sections $f \in H^0(\Pd, E)$ such that there exists point $p$ in the torus $T\coloneqq (\bC^*)^{C} \subset \Pd$ at which $f(p) = 0$ and the Jacobian of $f$ at $p$ drops rank, or equivalently, if the rank is less than $C = \sum_{i=1}^n (d_i - 1)$. (Note that the $\ma$-discriminant is the Zariski closure of $D_\ma^\circ$.) By the Jacobian criterion, since $\Pd$ is smooth of dimension $C$ and $E$ has rank $C$, the Jacobian drops rank exactly when $Z(f)$ is nonreduced or has a positive-dimensional component. Thus $D_\ma^\circ \subseteq \cone(\Delta(E))$, and taking closures of both sides gives the desired inclusion.
        
    To show that $D_\ma \supseteq \cone(\Delta(E))$, we first rule out the possibility that a global section $f$ exists with $Z(f) = \varnothing$. We can apply \cite[Theorem 2.8(2)]{abo2026vector} to see that, for balanced games, the degree of the top Chern class is positive, which implies that each global section of $E$ has at least one zero by \cite[Lemma 42.40.1]{stacks-project}. 
    
    Fix a global section $f \in \cone(\Delta(E))$. By the argument in the previous paragraph, as well as the Jacobian criterion, we know that $f$ has a multiple zero at some point $p$, not necessarily in $T$. Let $\mathrm{GL}_a$ denote the group of invertible $a \times a$ matrices with complex entries. Observe that the group $G \coloneqq  \prod_{i=1}^n \mathrm{GL}_{d_i}$ acts linearly and transitively on the space $\Pd$, and the torus is an open dense subset of $\Pd$. Define $U_p\coloneqq\{g \in G \mid g \cdot p \in T\}$. Since $G$ is an irreducible variety, it follows that $U_p$ is dense in $G$. Now define the morphism $\phi_f\colon G \to H^0(\Pd, E)$ by $g \mapsto g \cdot f$. Observe that for any element $g \in U_p$, by the chain rule, $g \cdot f$ has a multiple zero at $g\cdot p \in T$. Thus, $\phi_f(U_p) \subseteq D_\ma$. Consider now $\phi_f^{-1}(D_\ma)$, which contains $U_p$ and is closed, so must be $G$. Thus, 
    \[f = e \cdot f = \phi_f(e) \in D_\ma.\qedhere\]
    \end{proof}
\noindent Therefore, the dimensions and degrees of irreducible components of $D_\ma$ agree with those of $\Delta(\bd)$.

\subsection{Applying Corollary \ref{cor:pokidkin} to the strictly balanced case}

\begin{lemma}\label{lem:strictly-balanced-BK-tuple}
    In the strictly balanced case, namely when $d_n - 1 < \sum_{i=1}^{n-1} (d_i - 1)$, we have that $\ma$ is an irreducible BK-tuple.
\end{lemma}
\begin{proof}
    We can perform casework on the possible subtuples $\mb$ of $\ma$ as follows:

    Case 1: $\mb$ contains some $\Delta^{(i)}$ and $\Delta^{(j)}$ with $i \ne j$. In this case, the Minkowski sum is $\sum_{i = 1}^n (d_i -1)$-dimensional, so any proper subtuple $\mb$ of $\ma$ has positive defect and $\ma$ itself has defect 0. 

    Case 2: $\mb$ consists of some number of copies of $\Delta^{(i)}$. In this case, the defect is minimized when all the copies of $\Delta^{(i)}$ are included. I claim that, for any $i$, when $\mb$ consists of all copies of $\Delta^{(i)}$, the defect is still positive. The dimension of the Minkowski sum of any positive number of copies of $\Delta^{(i)}$ equals the dimension of a single copy of $\Delta^{(i)}$, which is a product of simplices and has dimension $\sum_{j \ne i} (d_j - 1)$. Then $|I| = d_i - 1$ and
    \[
        \delta(\mb) = \sum_{j \ne i} (d_j - 1) - (d_i - 1) = \sum_{j=1}^n (d_j- 1) - 2(d_i - 1) > 2(d_n - 1) - 2(d_i - 1) > 0\,,
    \]
    where the second-to-last inequality holds because $d_n - 1 < \sum_{i=1}^{n-1}(d_i - 1)$.
    
    Thus, we have verified that $\ma$ is independent and irreducible. Since it is defect 0, we conclude that $\ma$ is a BK-tuple.
\end{proof}

We now verify that $\ma$ is nonlinear using \cite[Theorem 3.4]{pokidkin}.

\begin{theorem}[{\cite[Theorem 3.4]{pokidkin}}]
The $\ma$-discriminant is a variety of codimension one if the tuple $\ma$ is
a nonlinear irreducible BK-tuple, and of codimension two if the tuple $\ma$ is a linear irreducible BK-tuple.
\end{theorem}

If $\ma$ were linear, the Nash $\ma$-discriminant would have codimension 2, but \cite[Theorem 3.26]{abo2026vector} shows that the $\ma$-discriminant is actually one-codimensional. Thus, $\ma$ is nonlinear. By \cite[Corollary 3.6]{pokidkin}, $\nabla_{\!\ka}$ coincides with $\Delta(\bd)$. By Corollary \ref{cor:discriminant-strictly-balanced}, we obtain the degree of $\Delta(\bd)$ for strictly balanced game formats $\bd$.

\begin{corollary}\label{cor:final-discriminant-strictly-balanced}
For strictly balanced game formats $\bd$, the Nash discriminant variety $\Delta(\bd)$ is an irreducible hypersurface. Its degree is the coefficient of $\prod_{i=1}^n z_i^{d_i-1}$ in 
\begin{equation}
    \left(C + \sum_{i=1}^n(C - 2d_i + 1)z_i\right) \prod_{i=1}^n(1 + \hat{z}_i)^{d_i-1}\,.
\end{equation}
\end{corollary}

\subsection{Applying Theorem \ref{thm:pokidkin} to the on-boundary case}

Define $\ma' \coloneqq (\Delta^{(n)},\dots,\Delta^{(n)})$, where $\Delta^{(n)}$ is repeated $d_n-1$ times.
From an analysis quite similar to that in Lemma \ref{lem:strictly-balanced-BK-tuple}, it is apparent that the only BK-subtuples of $\ma$ are $\ma$ and $\ma'$. The tuple $\ma'$ appears in Case 2, when $i = n$ and the number of copies of $\Delta^{(i)}$ is maximized.
By Theorem \ref{thm:pokidkin}, the number of BK-subtuples of $\ma$ gives an upper bound for the number of irreducible components of the $\ma$-discriminant. In this case, the upper bound is 2. By \cite[Theorem 3.26]{abo2026vector}, the lower bound is 2. Thus the $\ma$-discriminant has exactly two components. As established in \cite[Theorem 3.26]{abo2026vector}, one component has codimension 1, while the other has codimension at least 2; by Theorem \ref{thm:pokidkin}, the codimension must be exactly 2. 

To summarize, the Nash $\ma$-discriminant is the union of two irreducible components, one of codimension 1 and one of codimension 2.

Now, we may identify the affine cone over the codimension-two Cayley discriminant of prelinear simple BK-subtuple $\ma$ with $\nabla_{\!\ka}$. Thus, using the degree formulas in \cite[Theorem 3.26]{abo2026vector} and in Corollary \ref{cor:delta-2-nonzero}, we obtain a full description of the irreducible components and degrees of the Nash discriminant variety $\Delta(\bd)$.

\begin{corollary}\label{cor:discriminant-on-boundary}
For on-boundary game formats $\bd$, the Nash discriminant variety $\Delta(\bd)$ has two irreducible components, one of codimension 1 and one of codimension 2. The first has degree
\[
    \frac{(d_n - 1)!}{(d_1 - 1)! \cdots (d_{n-1} - 1)!} \left(d_n - 1 + \sum_{i=1}^{n-1} (d_i - 1)(d_n - d_i - 1)\right)
\]
and the second has degree
\[
    \frac{1}{2} \binom{d_n - 1}{d_1 - 1,\dots,d_{n-1} - 1}\left(d_n(d_n - 1) + (4d_n - 1)e_2(\bd-\bone) + 2 e_2(\bd-\bone)^2\right)
\]
where $e_2(\bd-\bone) = \sum_{1 \le i < j \le n - 1}(d_i - 1)(d_j - 1)$.
\end{corollary}

\section*{Acknowledgements}
We are grateful to Hirotachi Abo and Bernd Sturmfels for suggesting that we study this problem from the perspective of $\ka$-discriminants and for offering helpful suggestions during the research process. K.T. thanks the Max Planck Institute for Mathematics in the Sciences for their hospitality and access to the computing server, and Jonas Iskander for helpful conversations. We used generative AI (ChatGPT) to convert human-written SageMath code for the Nashmatrix into Julia; the Julia code is available in the Zenodo repository \cite{Zenodo}.

\bibliographystyle{alphaurl}
\bibliography{biblio}
\end{document}